\documentclass[11pt]{amsart}

\usepackage{amscd}
\usepackage{amsmath, amssymb}
\usepackage{amsfonts}
\usepackage{enumerate}

\begin{document}
\newcounter{remark}
\newcounter{theor}
\setcounter{remark}{0}
\setcounter{theor}{1}
\newtheorem{claim}{Claim}
\newtheorem{theorem}{Theorem}[section]
\newtheorem{lemma}[theorem]{Lemma}
\newtheorem{corollary}[theorem]{Corollary}
\newtheorem{proposition}[theorem]{Proposition}
\newtheorem{question}{question}[section]
\newtheorem{defn}{Definition}[theor]
\numberwithin{equation}{section}

\title[The parabolic Monge-Amp\`{e}re equation]{The parabolic Monge-Amp\`{e}re equation on compact almost Hermitian manifolds}
\author{Jianchun Chu}
\address{School of Mathematical Sciences, Peking University, Yiheyuan Road 5, Beijing, P.R.China, 100871}
\email{chujianchun@pku.edu.cn}

\begin{abstract}
We prove the long time existence and uniqueness of solutions to the parabolic Monge-Amp\`{e}re equation on compact almost Hermitian manifolds. We also show that the normalization of solution converges to a smooth function in $C^{\infty}$ topology as $t\rightarrow\infty$. Up to scaling, the limit function is a solution of the Monge-Amp\`{e}re equation. This gives a parabolic proof of existence of solutions to the Monge-Amp\`{e}re equation on almost Hermitian manifolds.
\end{abstract}
\maketitle

\section{Introduction}
Let $(M,\omega,J)$ be an almost Hermitian manifold of real dimension $2n$. And we use $g$ to denote the corresponding Riemannian metric. For a smooth real-valued function $F$ on $M$, we consider the Monge-Amp\`{e}re equation
\begin{equation}\label{Complex Monge-Ampere Equation}
\left\{ \begin{array}{ll}
\ (\omega+\sqrt{-1}\partial\overline{\partial}\varphi)^{n}=e^{F}\omega^{n}\\[1mm]
\ \tilde{\omega}=\omega+\sqrt{-1}\partial\bar{\partial}\varphi>0\\[1mm]
\ \sup_{M}\varphi=0
\end{array}\right.
\end{equation}
and the parabolic Monge-Amp\`{e}re equation
\begin{equation}\label{Parabolic Monge-Ampere Equation}
\left\{ \begin{array}{ll}
\ \frac{\partial\varphi}{\partial t}=\log\frac{(\omega+\sqrt{-1}\partial\overline{\partial}\varphi)^{n}}{\omega^{n}}-F \\[1mm]
\ \varphi(\cdot,0)=\varphi_{0}\\[1mm]
\ \tilde{\omega}=\omega+\sqrt{-1}\partial\bar{\partial}\varphi>0
\end{array}\right.
\end{equation}
where $\sqrt{-1}\partial\overline{\partial}\varphi=\frac{1}{2}(dJd\varphi)^{(1,1)}$ and $\varphi_{0}$ is a smooth real-valued function such that $\omega+\sqrt{-1}\partial\overline{\partial}\varphi_{0}>0$.

The Monge-Amp\`{e}re equation (\ref{Complex Monge-Ampere Equation}) plays an important role in geometry. When $(M,\omega,J)$ is a compact K\"{a}hler manifold, Calabi \cite{Calabi} presented his famous conjecture and transformed this problem into (\ref{Complex Monge-Ampere Equation}). By using the maximum principle, Calabi \cite{Calabi} proved the uniqueness of solutions to (\ref{Complex Monge-Ampere Equation}). In \cite{Yau}, Yau solved Calabi's conjecture by proving existence of solutions to (\ref{Complex Monge-Ampere Equation}) when $F$ satisfies $\int_{M}e^{F}\omega^{n}=\int_{M}\omega^{n}$.

When $(M,\omega,J)$ is a compact Hermitian manifold, (\ref{Complex Monge-Ampere Equation}) has been studied under some assumptions on $\omega$ (see \cite{Chu15,GL10,Hanani,TW10a}). For general $\omega$, up to adding a unique constant to $F$, the existence and uniqueness of solutions were proved by Cherrier \cite{Cherrier} for $n=2$ (and under assumption $d(\omega^{n-1})=0$ when $n>2$) and by Tosatti-Weinkove \cite{TW10b} for any dimensions.

When $(M,\omega,J)$ is a compact almost Hermitian manifold, Chu-Tosatti-Weinkove \cite{CTW} proved the existence and uniqueness of solutions to (\ref{Complex Monge-Ampere Equation}), up to adding a unique constant to $F$.

There are many results of complex Monge-Amp\`{e}re equation and complex Monge-Amp\`{e}re type equation, we refer the reader to \cite{Chu16,DZZ,GL12,GS,Li,PSS,Sun16,Szekelyhidi,STW,Tian83,Tian14,TWWY,TW13a,TW13b,TWY,Wang,Weinkove07,Zhang}.

For the parabolic Monge-Amp\`{e}re equation (\ref{Parabolic Monge-Ampere Equation}), when $(M,\omega,J)$ is a compact K\"{a}hler manifold, Cao \cite{Cao} proved that there exists a smooth solution for all time (long time existence) and the normalization of this solution converges smoothly to the solution of complex Monge-Amp\`{e}re equation. When $(M,\omega,J)$ is a compact Hermitian manifold, similar results were proved by Gill \cite{Gill}. And Sun \cite{Sun15} proved the analogous results for the parabolic Monge-Amp\`{e}re type equation.

As we can see, the results in \cite{Cao,Gill,Sun15} were proved when the almost complex structure $J$ is integrable. For non-integrable almost complex structure, we prove the following result in this paper.

\begin{theorem}\label{Main Theorem}
Let $(M,\omega,J)$ be a compact almost Hermitian manifold of real dimension $2n$. For the parabolic Monge-Amp\`{e}re equation (\ref{Parabolic Monge-Ampere Equation}) on $(M,\omega,J)$, we have
\begin{enumerate}[(1)]
  \item There exists a unique smooth solution $\varphi$ for $t\in[0,\infty)$.\\
  \item Let $\tilde{\varphi}$ be the normalization of $\varphi$, i.e.,
        \begin{equation*}
        \tilde{\varphi}=\varphi-\int_{M}\varphi~\omega^{n}.
        \end{equation*}
        Then $\tilde{\varphi}$ converges smoothly to a function $\tilde{\varphi}_{\infty}$ as $t\rightarrow\infty$. And $\tilde{\varphi}_{\infty}$ is the unique solution of (\ref{Complex Monge-Ampere Equation}) on $(M,\omega,J)$, up to adding a unique real constant $b$ to $F$.
\end{enumerate}
\end{theorem}

The organization of this paper is as follows: In Section 2, we introduce some notations and basic results which we use in this paper. In Section 3 through 5, we derive some estimates of $\varphi$. In Section 6, we use these estimates to prove (1) of Theorem \ref{Main Theorem}. In Section 7, we build up the Harnack inequality for positive solutions to the heat type equation on compact almost Hermitian manifold $(M,\omega,J)$, which is the generalized version of Theorem 2.2 in \cite{LY}. In Section 8, we apply this Harnack inequality to prove (2) of Theorem \ref{Main Theorem}.

\bigskip
\noindent\textbf{Acknowledgments:} The author would like to thank his advisor Professor Gang Tian for encouragement and support. The author would also like to thank Professor Valentino Tosatti and Professor Ben Weinkove for suggesting this problem and helpful suggestions. This work was carried out while the author was visiting the Department of Mathematics at Northwestern University, supported by the China Scholarship Council (File No. 201506010010). The author would like to thank the China Scholarship Council for supporting this visiting. The author would also like to thank the Department of Mathematics at Northwestern University for its hospitality and for providing a good academic environment.

\section{Preliminaries}
Let $M$ be a compact manifold of real dimension $2n$ and $J$ be an almost complex structure on $M$. Then we have decomposition $T_{\mathbb{C}}M=T_{\mathbb{C}}^{(1,0)}M\oplus T_{\mathbb{C}}^{(0,1)}M$, where $T_{\mathbb{C}}M$ is the complexification of $TM$, $T_{\mathbb{C}}^{(1,0)}M$ and $T_{\mathbb{C}}^{(0,1)}M$ are the $\pm\sqrt{-1}$ eigenspaces of $J$. For any $1$ form $\alpha$ on $M$, we define
\begin{equation*}
J\alpha(V)=-\alpha(JV)
\end{equation*}
for $V\in TM$. By this definition, the complexified cotangent space $T_{\mathbb{C}}^{*}M$ has the similar decomposition as $T_{\mathbb{C}}M$. By this decomposition, we introduce the definitions of $(1,0)$ form and $(0,1)$ form. More generally, we can also introduce the definition of $(p,q)$ form. For any $(p,q)$ form $\beta$, we define $\partial\beta$ and $\overline{\partial}\beta$ by $\partial\beta=(d\beta)^{(p+1,q)}$ and $\overline{\partial}\beta=(d\beta)^{(p,q+1)}$. It then follows that, for any smooth function $f$ on $M$, we have
\begin{equation*}
\begin{split}
(dJdf)^{(1,1)} & = (-\sqrt{-1}d\partial f+\sqrt{-1}d\overline{\partial} f)^{(1,1)}\\
& = 2\sqrt{-1}\partial\overline{\partial} f.
\end{split}
\end{equation*}
This is the reason why we use $\sqrt{-1}\partial\overline{\partial}\varphi$ to denote $\frac{1}{2}(dJdf)^{(1,1)}$ in Section 1. We also have the following formula (see e.g. \cite[(2.5)]{HL})
\begin{equation*}
(\partial\overline{\partial}f)(V_{1},\overline{V}_{2})=V_{1}\overline{V}_{2}(f)-[V_{1},\overline{V}_{2}]^{(0,1)}(f)
\end{equation*}
for any $V_{1},V_{2}\in T_{\mathbb{C}}^{(1,0)}M$.

Let $g$ be a Riemannian metric on $M$. We recall that $(M,g,J)$ is an almost Hermitian manifold if $g$ and $J$ are compatible, i.e.,
\begin{equation*}
g(JV_{1},JV_{2})=g(V_{1},V_{2})
\end{equation*}
for any $V_{1},V_{2}\in TM$. We can define the corresponding $(1,1)$ form
\begin{equation*}
\omega(V_{1},V_{2})=g(JV_{1},V_{2})
\end{equation*}
for any $V_{1},V_{2}\in TM$. It is clear that
\begin{equation*}
g(V_{1},V_{2})=\omega(V_{1},JV_{2}).
\end{equation*}
And $g$ is called the corresponding Riemannian metric of $(M,\omega,J)$. For convenience, we often use $(M,\omega,J)$ to denote $(M,g,J)$.

For (\ref{Parabolic Monge-Ampere Equation}), we use $\tilde{\omega}$ to denote $\omega+\sqrt{-1}\partial\overline{\partial}\varphi$. Here we omit time $t$ when no confusion will arise. Let $\tilde{g}$ be the corresponding Riemannian metric of $(M,\tilde{\omega},J)$.

We shall use the following notions, for a smooth function $f$ on $M$ and local frame $\{e_{i}\}_{i=1}^{n}$ for $T_{\mathbb{C}}^{(1,0)}M$,
\begin{equation*}
|\partial f|_{g}^{2}=g^{i\overline{j}}e_{i}(f)\overline{e}_{j}(f) \quad \text{and} \quad |\partial f|_{\tilde{g}}^{2}=\tilde{g}^{i\overline{j}}e_{i}(f)\overline{e}_{j}(f).
\end{equation*}
For convenience, we often use $f_{i}$ and $f_{\overline{i}}$ to denote $e_{i}(f)$ and $\overline{e}_{i}(f)$, respectively. As in \cite{CTW,Plis}, we define a operator
\begin{equation}\label{Definition of operator L}
\begin{split}
L(f) & = \tilde{g}^{i\overline{j}}\partial\overline{\partial}f(e_{i},\overline{e}_{j})\\
& = \tilde{g}^{i\overline{j}}\left(e_{i}\overline{e}_{j}(f)-[e_{i},\overline{e}_{j}]^{(0,1)}(f)\right).
\end{split}
\end{equation}
It is clear that $L$ is a second order elliptic operator. Since $L$ is the linearized operator of (\ref{Parabolic Monge-Ampere Equation}), by standard parabolic theory, there exists a smooth solution $\varphi$ to (\ref{Parabolic Monge-Ampere Equation}) on $[0,T)$, where $[0,T)$ is the maximal time interval and $T\in(0,\infty]$.

In this paper, we say a constant is uniform if it depends only on $(M,\omega,J)$, $F$ and $\varphi_{0}$. And we use often use $C$ to denote a uniform constant, which may differ from line to line. We shall point out that we use Einstein notation convention throughout this paper. Sometimes, we will include the summation for clarity.

\section{Oscillation estimate}
In this section, we prove the oscillation estimate of solution $\varphi$ to (\ref{Parabolic Monge-Ampere Equation}). First, we need the following lemma.
\begin{lemma}\label{varphi t estimate}
Let $\varphi$ be the solution of (\ref{Parabolic Monge-Ampere Equation}). Then we have
\begin{equation*}
\sup_{M\times[0,T)}\left|\frac{\partial\varphi}{\partial t}(x,t)\right|\leq \left\|\log\frac{(\omega+\sqrt{-1}\partial\overline{\partial}\varphi_{0})^{n}}{\omega^{n}}\right\|_{L^{\infty}(M)}+ \|F\|_{L^{\infty}(M)},
\end{equation*}
where $[0,T)$ is the maximal time interval of solution $\varphi$.
\end{lemma}

\begin{proof}
Differentiating (\ref{Parabolic Monge-Ampere Equation}) with respect to $t$, we obtain
\begin{equation*}
\left(L-\frac{\partial}{\partial t}\right)\frac{\partial\varphi}{\partial t}=0,
\end{equation*}
where $L$ is defined by (\ref{Definition of operator L}). By the maximum principle, it is clear that
\begin{equation}\label{varphi t estimate equation 1}
\sup_{M\times[0,T)}\left|\frac{\partial\varphi}{\partial t}(x,t)\right|\leq\sup_{M}\left|\frac{\partial\varphi}{\partial t}(x,0)\right|.
\end{equation}
By (\ref{Parabolic Monge-Ampere Equation}), we have
\begin{equation}\label{varphi t estimate equation 2}
\frac{\partial\varphi}{\partial t}(x,0)=\log\frac{(\omega+\sqrt{-1}\partial\overline{\partial}\varphi_{0})^{n}}{\omega^{n}}-F(x).
\end{equation}
Combining (\ref{varphi t estimate equation 1}) and (\ref{varphi t estimate equation 2}), we get
\begin{equation*}
\sup_{M\times[0,T)}\left|\frac{\partial\varphi}{\partial t}(x,t)\right|\leq
\left\|\log\frac{(\omega+\sqrt{-1}\partial\overline{\partial}\varphi_{0})^{n}}{\omega^{n}}\right\|_{L^{\infty}(M)}+ \|F\|_{L^{\infty}(M)}.
\end{equation*}
\end{proof}

Next, we use Lemma \ref{varphi t estimate} to prove the oscillation estimate.
\begin{proposition}\label{Oscillation estimate}
Let $\varphi$ be the solution of (\ref{Parabolic Monge-Ampere Equation}). There exists a constant $C$ depending only on $(M,\omega,J)$, $F$ and $\varphi_{0}$ such that
\begin{equation*}
\sup_{M\times[0,T)}|\tilde{\varphi}(x,t)| \leq \sup_{t\in[0,T)}\left(\sup_{x\in M}\varphi(x,t)-\inf_{x\in M}\varphi(x,t)\right)\leq C.
\end{equation*}
where $\tilde{\varphi}=\varphi-\int_{M}\varphi~\omega^{n}$ and $[0,T)$ is the maximal time interval of solution $\varphi$.
\end{proposition}

\begin{proof}
First, (\ref{Parabolic Monge-Ampere Equation}) can be written as
\begin{equation*}
(\omega+\sqrt{-1}\partial\overline{\partial}\varphi)^{n}=e^{\tilde{F}}\omega^{n},
\end{equation*}
where $\tilde{F}=F+\frac{\partial\varphi}{\partial t}$. By Proposition 3.1 in \cite{CTW}, there exists a constant $C$ depending only on $(M,\omega,J)$ and upper bound of $\sup_{M\times[0,T)}|\tilde{F}(x,t)|$ such that, for any $t\in[0,T)$,
\begin{equation}\label{Oscillation estimate equation 1}
\sup_{x\in M}\varphi(x,t)-\inf_{x\in M}\varphi(x,t)\leq C(M,\omega,J,\sup_{M\times[0,T)}|\tilde{F}(x,t)|).
\end{equation}
Thus, by Lemma \ref{varphi t estimate}, we have
\begin{equation}\label{Oscillation estimate equation 2}
\begin{split}
\sup_{M\times[0,T)}|\tilde{F}(x,t)| & \leq \|F\|_{L^{\infty}(M)}+\sup_{M\times[0,T)}\left|\frac{\partial\varphi}{\partial t}(x,t)\right|\\
& \leq 2\|F\|_{L^{\infty}(M)}+\left\|\log\frac{(\omega+\sqrt{-1}\partial\overline{\partial}\varphi_{0})^{n}}{\omega^{n}}\right\|_{L^{\infty}(M)} .
\end{split}
\end{equation}
Combining (\ref{Oscillation estimate equation 1}) and (\ref{Oscillation estimate equation 2}), for any $t\in[0,T)$, it is clear that
\begin{equation*}
\sup_{x\in M}\varphi(x,t)-\inf_{x\in M}\varphi(x,t)\leq C,
\end{equation*}
for a uniform constant $C$. Hence, by the definition of $\tilde{\varphi}$, we complete the proof.
\end{proof}

\section{First order estimate}
In this section, we prove the first order estimate of solution $\varphi$ to (\ref{Parabolic Monge-Ampere Equation}).
\begin{proposition}\label{First order estimate}
Let $\varphi$ be the solution of (\ref{Parabolic Monge-Ampere Equation}). There exists a constant $C$ depending only on $(M,\omega,J)$, $F$ and $\varphi_{0}$ such that
\begin{equation*}
\sup_{M\times[0,T)}|\partial\varphi|_{g}^{2}(x,t)\leq C,
\end{equation*}
where $[0,T)$ is the maximal time interval of solution $\varphi$.
\end{proposition}

\begin{proof}
We consider the quantity $Q=e^{f(\tilde{\varphi})}|\partial\varphi|_{g}^{2}$, where $\tilde{\varphi}=\varphi-\int_{M}\varphi~\omega^{n}$ and $f$ is to be determined later. For any $T'\in[0,T)$, we assume
\begin{equation*}
\max_{M\times[0,T']}Q(x,t)=Q(x_{0},t_{0}),
\end{equation*}
where $(x_{0},t_{0})\in M\times[0,T']$. Since $g$ is compatible with $J$, around $x_{0}$, we can find a local unitary frame $\{e_{i}\}_{i=1}^{n}$ (with respect to $g$) for $T_{\mathbb{C}}^{(1,0)}M$ such that $\tilde{g}_{i\overline{j}}(x_{0},t_{0})$ is diagonal.

By the maximum principle, at $(x_{0},t_{0})$, we have
\begin{equation}\label{First order estimate equation 1}
\begin{split}
0 & \geq \left(L-\frac{\partial}{\partial t}\right)Q\\
  & = \left(L-\frac{\partial}{\partial t}\right)\left(e^{f}|\partial\varphi|_{g}^{2}\right)\\
  & = e^{f}\left(L-\frac{\partial}{\partial t}\right)|\partial\varphi|_{g}^{2}+2\text{Re}\left(\tilde{g}^{i\overline{i}}e_{i}(|\partial\varphi|_{g}^{2})\overline{e}_{i}(e^{f})\right)+|\partial\varphi|_{g}^{2}\left(L-\frac{\partial}{\partial t}\right)e^{f}.
\end{split}
\end{equation}
For the first term of (\ref{First order estimate equation 1}), by direct calculation, we obtain
\begin{equation}\label{First order estimate equation 2}
\begin{split}
\left(L-\frac{\partial}{\partial t}\right)|\partial\varphi|_{g}^{2}
= &~~ \sum_{k}\tilde{g}^{i\overline{i}}\left(|e_{i}e_{k}(\varphi)|^{2}+|e_{i}\overline{e}_{k}(\varphi)|^{2}\right)\\
&~~ +\sum_{k}\varphi_{k}\left(\tilde{g}^{i\overline{i}}e_{i}\overline{e}_{i}\overline{e}_{k}(\varphi)
-\tilde{g}^{i\overline{i}}[e_{i},\overline{e}_{i}]^{(0,1)}\overline{e}_{k}(\varphi)-\left(\frac{\partial\varphi}{\partial t}\right)_{\overline{k}}\right)\\
&~~ +\sum_{k}\varphi_{\overline{k}}\left(\tilde{g}^{i\overline{i}}e_{i}\overline{e}_{i}e_{k}(\varphi)
-\tilde{g}^{i\overline{i}}[e_{i},\overline{e}_{i}]^{(0,1)}e_{k}(\varphi)-\left(\frac{\partial\varphi}{\partial t}\right)_{k}\right).
\end{split}
\end{equation}
In order to deal with the second and third terms of (\ref{First order estimate equation 2}), we compute
\begin{equation}\label{First order estimate equation 3}
\begin{split}
&~~\sum_{k}\varphi_{k}\tilde{g}^{i\overline{i}}\left(e_{i}\overline{e}_{i}\overline{e}_{k}(\varphi)
-[e_{i},\overline{e}_{i}]^{(0,1)}\overline{e}_{k}(\varphi)\right)\\
=&~~\sum_{k}\varphi_{k}\tilde{g}^{i\overline{i}}\left(\overline{e}_{k}e_{i}\overline{e}_{i}(\varphi)-\overline{e}_{k}[e_{i},\overline{e}_{i}]^{(0,1)}(\varphi)\right)\\
&~~+\sum_{k}\varphi_{k}\tilde{g}^{i\overline{i}}\left(-\overline{e}_{i}[\overline{e}_{k},e_{i}](\varphi)-e_{i}[\overline{e}_{k},\overline{e}_{i}](\varphi)+
[\overline{e}_{i},[\overline{e}_{k},e_{i}]](\varphi)+[\overline{e}_{k},[e_{i},\overline{e}_{i}]^{(0,1)}](\varphi)\right)\\
\geq&~~\sum_{k}\varphi_{k}\tilde{g}^{i\overline{i}}\left(\overline{e}_{k}e_{i}\overline{e}_{i}(\varphi)-\overline{e}_{k}[e_{i},\overline{e}_{i}]^{(0,1)}(\varphi)\right)\\
&~~-C|\partial\varphi|_{g}\sum_{k}\tilde{g}^{i\overline{i}}\left(|e_{i}e_{k}(\varphi)|+|e_{i}\overline{e}_{k}(\varphi)|\right)-C|\partial\varphi|_{g}^{2}\sum_{i}\tilde{g}^{i\overline{i}},
\end{split}
\end{equation}
for a uniform constant $C$. Now, applying $\overline{e}_{k}$ to (\ref{Parabolic Monge-Ampere Equation}), we get
\begin{equation}\label{First order estimate equation 4}
\left(\frac{\partial\varphi}{\partial t}\right)_{\overline{k}}
=\tilde{g}^{i\overline{i}}\left(\overline{e}_{k}e_{i}\overline{e}_{i}(\varphi)-\overline{e}_{k}[e_{i},\overline{e}_{i}]^{(0,1)}(\varphi)\right)-F_{\overline{k}}.
\end{equation}
Combining (\ref{First order estimate equation 3}) and (\ref{First order estimate equation 4}), it is clear that
\begin{equation}\label{First order estimate equation 5}
\begin{split}
&~~\sum_{k}\varphi_{k}\left(\tilde{g}^{i\overline{i}}e_{i}\overline{e}_{i}\overline{e}_{k}(\varphi)
-\tilde{g}^{i\overline{i}}[e_{i},\overline{e}_{i}]^{(0,1)}\overline{e}_{k}(\varphi)-\left(\frac{\partial\varphi}{\partial t}\right)_{\overline{k}}\right)\\
\geq&~~\sum_{k}\varphi_{k}F_{\overline{k}}-C|\partial\varphi|_{g}\sum_{k}\tilde{g}^{i\overline{i}}\left(|e_{i}e_{k}(\varphi)|+|e_{i}\overline{e}_{k}(\varphi)|\right)-C|\partial\varphi|_{g}^{2}\sum_{i}\tilde{g}^{i\overline{i}}.
\end{split}
\end{equation}
Similarly, we have
\begin{equation}\label{First order estimate equation 6}
\begin{split}
&~~\sum_{k}\varphi_{\overline{k}}\left(\tilde{g}^{i\overline{i}}e_{i}\overline{e}_{i}e_{k}(\varphi)
-\tilde{g}^{i\overline{i}}[e_{i},\overline{e}_{i}]^{(0,1)}e_{k}(\varphi)-\left(\frac{\partial\varphi}{\partial t}\right)_{k}\right)\\
\geq&~~\sum_{k}\varphi_{\overline{k}}F_{k}-C|\partial\varphi|_{g}\sum_{k}\tilde{g}^{i\overline{i}}\left(|e_{i}e_{k}(\varphi)|+|e_{i}\overline{e}_{k}(\varphi)|\right)-C|\partial\varphi|_{g}^{2}\sum_{i}\tilde{g}^{i\overline{i}}.
\end{split}
\end{equation}
Combining (\ref{First order estimate equation 2}), (\ref{First order estimate equation 5}), (\ref{First order estimate equation 6}) and the Cauchy-Schwarz inequality, for any $\epsilon\in(0,\frac{1}{2}]$, we obtain
\begin{equation}\label{First order estimate equation 7}
\begin{split}
\left(L-\frac{\partial}{\partial t}\right)|\partial\varphi|_{g}^{2}\geq
&~~(1-\epsilon)\sum_{k}\tilde{g}^{i\overline{i}}\left(|e_{i}e_{k}(\varphi)|^{2}+|e_{i}\overline{e}_{k}(\varphi)|^{2}\right)\\
&~~-\frac{C}{\epsilon}|\partial\varphi|_{g}^{2}\sum_{i}\tilde{g}^{i\overline{i}}+2\text{Re}\left(\sum_{k}\varphi_{k}F_{\overline{k}}\right).
\end{split}
\end{equation}
For the second term of (\ref{First order estimate equation 1}), since $\overline{e}_{i}(\tilde{\varphi})=\overline{e}_{i}(\varphi)$, we have
\begin{equation}\label{First order estimate equation 8}
\begin{split}
&~~2\text{Re}\left(\tilde{g}^{i\overline{i}}e_{i}(|\partial\varphi|_{g}^{2})\overline{e}_{i}(e^{f})\right)\\
=&~~2\text{Re}\left(\sum_{k}\tilde{g}^{i\overline{i}}e^{f}f'\varphi_{\overline{i}}\varphi_{k}e_{i}\overline{e}_{k}(\varphi)\right)
+2\text{Re}\left(\sum_{k}\tilde{g}^{i\overline{i}}e^{f}f'\varphi_{\overline{i}}\varphi_{\overline{k}}e_{i}e_{k}(\varphi)\right)\\
=&~~2\text{Re}\left(\sum_{k}\tilde{g}^{i\overline{i}}e^{f}f'\varphi_{\overline{i}}\varphi_{k}\left(\tilde{g}_{i\overline{k}}-g_{i\overline{k}}+[e_{i},\overline{e}_{k}]^{(0,1)}(\varphi)\right)\right)
+2\text{Re}\left(\sum_{k}\tilde{g}^{i\overline{i}}e^{f}f'\varphi_{\overline{i}}\varphi_{\overline{k}}e_{i}e_{k}(\varphi)\right)\\
\geq&~~2e^{f}f'|\partial\varphi|_{g}^{2}-2e^{f}f'|\partial\varphi|_{\tilde{g}}^{2}-\epsilon e^{f}(f')^{2}|\partial\varphi|_{g}^{2}|\partial\varphi|_{\tilde{g}}^{2}-\frac{Ce^{f}}{\epsilon}|\partial\varphi|_{g}^{2}\sum_{i}\tilde{g}^{i\overline{i}}\\
&~~-(1+2\epsilon)e^{f}(f')^{2}|\partial\varphi|_{g}^{2}|\partial\varphi|_{\tilde{g}}^{2}-(1-\epsilon)e^{f}\sum_{k}\tilde{g}^{i\overline{i}}|e_{i}e_{k}(\varphi)|^{2}\\
\geq&~~2e^{f}f'|\partial\varphi|_{g}^{2}-2e^{f}f'|\partial\varphi|_{\tilde{g}}^{2}-(1+3\epsilon)e^{f}(f')^{2}|\partial\varphi|_{g}^{2}|\partial\varphi|_{\tilde{g}}^{2}-\frac{Ce^{f}}{\epsilon}|\partial\varphi|_{g}^{2}\sum_{i}\tilde{g}^{i\overline{i}}\\
&~~-(1-\epsilon)e^{f}\sum_{k}\tilde{g}^{i\overline{i}}|e_{i}e_{k}(\varphi)|^{2},
\end{split}
\end{equation}
where we used the Cauchy-Schwarz inequality for the second-to-last inequality. For the third term of (\ref{First order estimate equation 1}), it is clear that
\begin{equation}\label{First order estimate equation 9}
\begin{split}
\left(L-\frac{\partial}{\partial t}\right)e^{f}
& = e^{f}\left(f''+(f')^{2}\right)|\partial\tilde{\varphi}|_{\tilde{g}}^{2}+ne^{f}f'
-e^{f}f'\sum_{i}\tilde{g}^{i\overline{i}}-e^{f}f'\frac{\partial\tilde{\varphi}}{\partial t}\\
& = e^{f}\left(f''+(f')^{2}\right)|\partial\varphi|_{\tilde{g}}^{2}+ne^{f}f'
-e^{f}f'\sum_{i}\tilde{g}^{i\overline{i}}-e^{f}f'\left(\frac{\partial\varphi}{\partial t}-\int_{M}\frac{\partial\varphi}{\partial t}~\omega^{n}\right),
\end{split}
\end{equation}
where we used $|\partial\tilde{\varphi}|_{\tilde{g}}^{2}=|\partial\varphi|_{\tilde{g}}^{2}$. Plugging (\ref{First order estimate equation 7}), (\ref{First order estimate equation 8}) and (\ref{First order estimate equation 9}) into (\ref{First order estimate equation 1}), at $(x_{0},t_{0})$, we get

\begin{equation}\label{First order estimate equation 10}
\begin{split}
\left(L-\frac{\partial}{\partial t}\right)\left(e^{f}|\partial\varphi|_{g}^{2}\right) \geq &~~
e^{f}\left(f''-3\epsilon(f')^{2}\right)|\partial\varphi|_{g}^{2}|\partial\varphi|_{\tilde{g}}^{2}+e^{f}(-f'-\frac{C_{1}}{\epsilon})|\partial\varphi|_{g}^{2}\sum_{i}\tilde{g}^{i\overline{i}}\\
&~~ -e^{f}f'\left(\frac{\partial\varphi}{\partial t}-\int_{M}\frac{\partial\varphi}{\partial t}~\omega^{n}\right)|\partial\varphi|_{g}^{2}+2e^{f}\text{Re}\left(\sum_{k}\varphi_{k}F_{\overline{k}}\right)\\
&~~+(n+2)e^{f}f'|\partial\varphi|_{g}^{2}-2e^{f}f'|\partial\varphi|_{\tilde{g}}^{2},
\end{split}
\end{equation}
for a uniform constant $C_{1}$. Now, we define
\begin{equation*}
f(\tilde{\varphi})=\frac{1}{12C_{1}}e^{-12C_{1}(\tilde{\varphi}-\sup_{M\times[0,T)}\tilde{\varphi}-1)} \text{~~and~~} \epsilon=2C_{1}e^{12C_{1}(\tilde{\varphi}(x_{0},t_{0})-\sup_{M\times[0,T)}\tilde{\varphi}-1)}.
\end{equation*}
By Proposition \ref{Oscillation estimate}, there exists a uniform constant $C$ such that
\begin{equation}\label{First order estimate equation 11}
f''-3\epsilon(f')^{2}\geq C^{-1} \text{~~and~~}
-f'-\frac{C_{1}}{\epsilon}\geq C^{-1}.
\end{equation}
Combining (\ref{First order estimate equation 1}), (\ref{First order estimate equation 10}), (\ref{First order estimate equation 11}), $f'<0$ and Lemma \ref{varphi t estimate}, at $(x_{0},t_{0})$, we get
\begin{equation}\label{First order estimate equation 12}
0\geq C^{-1}|\partial\varphi|_{g}^{2}|\partial\varphi|_{\tilde{g}}^{2}
+C^{-1}|\partial\varphi|_{g}^{2}\sum_{i}\tilde{g}^{i\overline{i}}-C|\partial\varphi|_{g}^{2}-C.
\end{equation}
On the other hand, by (\ref{Parabolic Monge-Ampere Equation}) and Lemma \ref{varphi t estimate}, we have
\begin{equation*}
C\omega^{n}\geq(\omega+\sqrt{-1}\partial\overline{\partial}\varphi)^{n}\geq C^{-1}\omega^{n},
\end{equation*}
for a uniform constant $C$. It then follows that
\begin{equation*}
\sum_{i}\tilde{g}^{i\overline{i}}\geq\left(\frac{\sum_{i}\tilde{g}_{i\overline{i}}}{\prod_{j}\tilde{g}_{j\overline{j}}}\right)^{\frac{1}{n-1}} \geq C^{-1}\left(\sum_{i}\tilde{g}_{i\overline{i}}\right)^{\frac{1}{n-1}}\geq C^{-1}C(n)\sum_{i}\tilde{g}_{i\overline{i}}^{\frac{1}{n-1}},
\end{equation*}
where $C(n)$ is a constant depending only on $n$. It then follows that
\begin{equation}\label{First order estimate equation 13}
\begin{split}
|\partial\varphi|_{\tilde{g}}^{2}+\sum_{i}\tilde{g}^{i\overline{i}} & \geq C^{-1}\sum_{i}\left(\frac{|\varphi_{i}|^{2}}{\tilde{g}_{i\overline{i}}}+\tilde{g}_{i\overline{i}}^{\frac{1}{n-1}}\right)\\
& \geq C^{-1}\sum_{i}|\varphi_{i}|^{\frac{2}{n}}\\
& \geq C^{-1}|\partial\varphi|_{g}^{\frac{2}{n}},
\end{split}
\end{equation}
where we used Young's inequality in the second line. Combining (\ref{First order estimate equation 12}) and (\ref{First order estimate equation 13}), we get
\begin{equation*}
0\geq C^{-1}|\partial\varphi|_{g}^{2+\frac{2}{n}}(x_{0},t_{0})-C|\partial\varphi|_{g}^{2}(x_{0},t_{0})-C.
\end{equation*}
It then follows that
\begin{equation*}
|\partial\varphi|_{g}^{2}(x_{0},t_{0})\leq C.
\end{equation*}
Therefore, by the definition of $(x_{0},t_{0})$ and Proposition \ref{Oscillation estimate}, we have
\begin{equation*}
\max_{M\times[0,T']}Q(x,t)=Q(x_{0},t_{0})\leq C,
\end{equation*}
which implies
\begin{equation*}
\max_{M\times[0,T']}|\partial\varphi|_{g}^{2}(x,t)\leq C,
\end{equation*}
for a uniform constant $C$. Since $T'\in[0,T)$ is arbitrary, we complete the proof.
\end{proof}

\section{Second order estimate}
In this section, we use techniques developed in \cite[Section 5]{CTW} to prove the second order estimate of solution $\varphi$ to (\ref{Parabolic Monge-Ampere Equation}).
\begin{proposition}\label{Second order estimate}
Let $\varphi$ be the solution of (\ref{Parabolic Monge-Ampere Equation}). There exists a constant $C$ depending only on $(M,\omega,J)$, $F$ and $\varphi_{0}$ such that
\begin{equation*}
\sup_{M\times[0,T)}|\nabla^{2}\varphi|_{g}\leq C,
\end{equation*}
where $\nabla$ is the Levi-Civita connection with respect to $g$ and $[0,T)$ is the maximal time interval of solution $\varphi$.
\end{proposition}

\begin{proof}
Let $\lambda_{1}(\nabla^{2}\varphi)\geq\lambda_{2}(\nabla^{2}\varphi)\geq\cdots\geq\lambda_{2n}(\nabla^{2}\varphi)$ be the eigenvalues of $\nabla^{2}\varphi$. It then follows that $\Delta\varphi=\sum_{\alpha=1}^{2n}\lambda_{\alpha}(\nabla^{2}\varphi)$, where $\Delta$ is the Laplace-Beltrami operator of $(M,g)$. By $\omega+\sqrt{-1}\partial\overline{\partial}\varphi>0$, it is clear that
\begin{equation*}
\Delta^{\mathbb{C}}\varphi=\frac{n\sqrt{-1}\partial\overline{\partial}\varphi\wedge\omega^{n-1}}{\omega^{n}}>-n,
\end{equation*}
where $\Delta^{\mathbb{C}}$ is the canonical Laplacian of $(M,\omega,J)$. Since the difference between $\Delta\varphi$ and $2\Delta^{\mathbb{C}}\varphi$ just contains first order terms of $\varphi$ (see e.g. \cite[Lemma 3.2]{Tosatti}). By Proposition \ref{First order estimate}, we obtain the lower bound of $\Delta\varphi$, i.e.,
\begin{equation*}
\sum_{\alpha=1}^{2n}\lambda_{\alpha}(\nabla^{2}\varphi)=\Delta\varphi\geq-C,
\end{equation*}
for a uniform constant $C$. As a result, we have
\begin{equation}\label{lambda 1 control equation 1}
|\nabla^{2}\varphi|_{g}\leq C\max\left(\lambda_{1}(\nabla^{2}\varphi),0\right)+C.
\end{equation}
Hence, in order to get the upper bound of $|\nabla^{2}\varphi|_{g}$, it suffices to prove $\lambda_{1}(\nabla^{2}\varphi)$ is bounded from above. Now, on the set $\{(x,t)\in M\times[0,T)~|~\lambda_{1}(\nabla^{2}\varphi)(x,t)>0\}$ (if this set is empty, then we obtain the upper bound of $\lambda_{1}(\nabla^{2}\varphi)$ directly), we consider the following quantity
\begin{equation*}
Q=\log\lambda_{1}(\nabla^{2}\varphi)+h(|\partial\varphi|_{g}^{2})+e^{-A(\tilde{\varphi}-\sup_{M\times[0,T)}\tilde{\varphi})},
\end{equation*}
where
\begin{equation*}
h(s)=-\frac{1}{2}\log\left(\sup_{M\times[0,T)}|\partial\varphi|_{g}^{2}-s+1\right), \quad \tilde{\varphi}=\varphi-\int_{M}\varphi~\omega^{n}
\end{equation*}
and $A$ is a very large uniform constant to be determined later. By direct calculation and Proposition \ref{First order estimate}, it is clear that
\begin{equation}\label{Properties of h}
h''-2(h')^{2}=0 \quad \text{and} \quad C^{-1}\leq h' \leq C,
\end{equation}
for a uniform constant $C$. For any $T'\in(0,T)$, we assume that $Q$ achieves its maximum at $(x_{0},t_{0})$ on $\{(x,t)\in M\times[0,T']~|~\lambda_{1}(\nabla^{2}\varphi)(x,t)>0\}$. Now, around $x_{0}\in M$, we can find a local unitary frame $\{e_{i}\}_{i=1}^{n}$ (with respect to $g$) for $T_{\mathbb{C}}^{(1,0)}M$ such that, at $(x_{0},t_{0})$, $g_{i\overline{j}}=\delta_{ij}$, $\tilde{g}_{i\overline{j}}=\delta_{ij}\tilde{g}_{i\overline{i}}$ and
\begin{equation*}
\tilde{g}_{1\overline{1}}\geq\tilde{g}_{2\overline{2}}\geq\cdots\geq\tilde{g}_{n\overline{n}},
\end{equation*}
where $\tilde{g}=\tilde{g}(\cdot,t_{0})$. Since $g$ is a Riemannian metric, there exists a normal coordinate system $\{x^{\alpha}\}_{\alpha=1}^{2n}$ centered at $x_{0}$. Note that $g$ and $J$ are compatible, after a linear change of coordinates, at $x_{0}$, we can assume
\begin{equation*}
e_{1}=\frac{1}{\sqrt{2}}(\partial_{1}-\sqrt{-1}\partial_{2}),e_{2}=\frac{1}{\sqrt{2}}(\partial_{3}-\sqrt{-1}\partial_{4}),\cdots,e_{n}=\frac{1}{\sqrt{2}}(\partial_{2n-1}-\sqrt{-1}\partial_{2n})
\end{equation*}
and
\begin{equation}\label{normal coordinate}
\frac{\partial g_{\alpha\beta}}{\partial x^{\gamma}}=\frac{\partial g(\partial_{\alpha},\partial_{\beta})}{\partial x^{\gamma}}=0 \quad \text{for any $\alpha,\beta,\gamma=1,2,\cdots,2n$.}
\end{equation}
Let $V_{\beta}$ be the $g$-unit eigenvector of $\lambda_{\beta}(\nabla^{2}\varphi)(x_{0},t_{0})$ for $\beta=1,2,\cdots,2n$. Next, we extend the eigenvectors $V_{\beta}$ to be vector fields around $x_{0}$ as follows. For any point $x$ near $x_{0}$, we define
\begin{equation*}
V_{\beta}(x)=V_{\beta}^{\alpha}\partial_{\alpha}(x) \quad \text{for $\alpha,\beta=1,2,\cdots,2n$.}
\end{equation*}
Now, we want to use the maximum principle to the quantity $Q$ at $(x_{0},t_{0})$. However, $Q$ may be not smooth at $(x_{0},t_{0})$. In order to deal with this problem, we use a perturbation argument as in \cite{Szekelyhidi,STW}. In the coordinate system $\{x^{\alpha}\}_{\alpha=1}^{2n}$, we define
\begin{equation*}
\begin{split}
B & = B_{\alpha\beta}dx^{\alpha}\otimes dx^{\beta}\\
&= (\delta_{\alpha\beta}-V_{1}^{\alpha}V_{1}^{\beta})dx^{\alpha}\otimes dx^{\beta}.
\end{split}
\end{equation*}
Note that $V_{1}^{\alpha}$ and $V_{1}^{\beta}$ are constants. Next, we define
\begin{equation}\label{Definition of Phi}
\begin{split}
\Phi & = \Phi_{\beta}^{\alpha}\frac{\partial}{\partial x^{\alpha}}\otimes dx^{\beta}\\
& = \left(g^{\alpha\gamma}\varphi_{\gamma\beta}-g^{\alpha\gamma}B_{\gamma\beta}\right)\frac{\partial}{\partial x^{\alpha}}\otimes dx^{\beta},
\end{split}
\end{equation}
where $\varphi_{\gamma\beta}=\nabla^{2}\varphi(\partial_{\gamma},\partial_{\beta})$. Let $\lambda_{1}(\nabla^{2}\Phi)\geq\lambda_{2}(\nabla^{2}\Phi)\geq\cdots\geq\lambda_{2n}(\nabla^{2}\Phi)$ be the eigenvalues of $\Phi$. It then follows that $\lambda_{1}(\Phi)=\lambda_{1}(\nabla^{2}\varphi)$ at $(x_{0},t_{0})$ and $\lambda_{1}(\Phi)\leq\lambda_{1}(\nabla^{2}\varphi)$ near $(x_{0},t_{0})$. Most importantly, at $(x_{0},t_{0})$, the eigenspace of $\Phi$ corresponding to $\lambda_{1}(\Phi)$ has dimension 1, which implies $\lambda_{1}(\Phi)$ is smooth near $(x_{0},t_{0})$. Hence, we consider the perturbed quantity $\hat{Q}$ defined by
\begin{equation*}
\hat{Q}=\log\lambda_{1}(\Phi)+h(|\partial\varphi|_{g}^{2})+e^{-A(\tilde{\varphi}-\sup_{M\times[0,T)}\tilde{\varphi})}.
\end{equation*}
It is clear that $(x_{0},t_{0})$ is still a local maximum point of $\hat{Q}$. And for $\alpha=1,2,\cdots,2n$, $V_{\alpha}$ are still the eigenvectors of $\lambda_{\alpha}(\Phi)$ at $(x_{0},t_{0})$. By Proposition \ref{Oscillation estimate} and Proposition \ref{First order estimate}, in order to prove Proposition \ref{Second order estimate}, we just need to get the upper bound of $\lambda_{1}$ at $(x_{0},t_{0})$. In what follows, we use $\lambda_{\alpha}$ to denote $\lambda_{\alpha}(\Phi)$ for convenience. And we always use $C$ to denote a uniform constant, which may differ from line to line.

Without loss of generality, we assume $\lambda_{1}$ is very large at $(x_{0},t_{0})$. Thus, by (\ref{lambda 1 control equation 1}), we have
\begin{equation}\label{lambda 1 control equation 2}
|\nabla^{2}\varphi|_{g}(x_{0},t_{0})\leq C\lambda_{1}(x_{0},t_{0}).
\end{equation}
First, we have the following lemmas.
\begin{lemma}\label{second order computation}
At $(x_{0},t_{0})$, we have
\begin{equation}\label{second order computation equation 1}
\begin{split}
\left(L-\frac{\partial}{\partial t}\right)\lambda_{1} \geq &~~ 2\sum_{\alpha>1}\frac{\tilde{g}^{i\overline{i}}|e_{i}(\varphi_{V_{1}V_{\alpha}})|^{2}}{\lambda_{1}-\lambda_{\alpha}}+\tilde{g}^{p\overline{p}}\tilde{g}^{q\overline{q}}|V_{1}(\tilde{g}_{p\overline{q}})|^{2}\\
&~~ -2\tilde{g}^{i\overline{i}}[V_{1},e_{i}]V_{1}\overline{e}_{i}(\varphi)-2\tilde{g}^{i\overline{i}}[V_{1},\overline{e}_{i}]V_{1}e_{i}(\varphi)-C\lambda_{1}\sum_{i}\tilde{g}^{i\overline{i}},
\end{split}
\end{equation}
where $\varphi_{V_{\alpha}V_{\beta}}=\nabla^{2}\varphi(V_{\alpha},V_{\beta})$ for $\alpha,\beta=1,2,\cdots,2n$.
\end{lemma}

\begin{proof}
By (\ref{normal coordinate}), (\ref{Definition of Phi}) and the formula for the derivatives of $\lambda_{1}$ (see e.g. \cite[Lemma 5.2]{CTW}), at $(x_{0},t_{0})$, we compute
\begin{equation}\label{second order computation equation 2}
\begin{split}
&~~\left(L-\frac{\partial}{\partial t}\right)\lambda_{1}\\
=&~~\tilde{g}^{i\overline{i}}\sum_{\mu>1}\frac{V_{1}^{\alpha}V_{\mu}^{\beta}V_{\mu}^{\gamma}V_{1}^{\delta}+V_{\mu}^{\alpha}V_{1}^{\beta}V_{1}^{\gamma}V_{\mu}^{\delta}}{\lambda_{1}-\lambda_{\mu}}
e_{i}(\Phi_{\delta}^{\gamma})\overline{e}_{i}(\Phi_{\beta}^{\alpha})+\tilde{g}^{i\overline{i}}V_{1}^{\alpha}V_{1}^{\beta}e_{i}\overline{e}_{i}(\Phi_{\beta}^{\alpha})\\
&~~-\tilde{g}^{i\overline{i}}V_{1}^{\alpha}V_{1}^{\beta}[e_{i},\overline{e}_{i}]^{(0,1)}(\Phi_{\beta}^{\alpha})
-V_{1}^{\alpha}V_{1}^{\beta}\frac{\partial}{\partial t}\Phi_{\beta}^{\alpha}\\
= &~~2\sum_{\alpha>1}\frac{\tilde{g}^{i\overline{i}}|e_{i}(\varphi_{V_{1}V_{\alpha}})|^{2}}{\lambda_{1}-\lambda_{\alpha}}
+\tilde{g}^{i\overline{i}}e_{i}\overline{e}_{i}(\varphi_{V_{1}V_{1}})-\tilde{g}^{i\overline{i}}[e_{i},\overline{e}_{i}]^{(0,1)}(\varphi_{V_{1}V_{1}})\\
&~~ +\tilde{g}^{i\overline{i}}V_{1}^{\alpha}V_{1}^{\beta}\varphi_{\gamma\beta}e_{i}\overline{e}_{i}(g^{\alpha\gamma})
-\tilde{g}^{i\overline{i}}V_{1}^{\alpha}V_{1}^{\beta}B_{\gamma\beta}e_{i}\overline{e}_{i}(g^{\alpha\gamma})-\nabla^{2}\left(\frac{\partial\varphi}{\partial t}\right)(V_{1},V_{1})\\
\geq &~~ 2\sum_{\alpha>1}\frac{\tilde{g}^{i\overline{i}}|e_{i}(\varphi_{V_{1}V_{\alpha}})|^{2}}{\lambda_{1}-\lambda_{\alpha}}
+\tilde{g}^{i\overline{i}}e_{i}\overline{e}_{i}(\varphi_{V_{1}V_{1}})-\tilde{g}^{i\overline{i}}[e_{i},\overline{e}_{i}]^{(0,1)}(\varphi_{V_{1}V_{1}})\\
&~~ -C\lambda_{1}\sum_{i}\tilde{g}^{i\overline{i}}-\nabla^{2}\left(\frac{\partial\varphi}{\partial t}\right)(V_{1},V_{1}),
\end{split}
\end{equation}
where we used (\ref{lambda 1 control equation 2}) and $\lambda_{1}>>1$ at $(x_{0},t_{0})$ for the last inequality. By direct calculation, we have
\begin{equation}\label{second order computation equation 3}
\begin{split}
&~~\tilde{g}^{i\overline{i}}e_{i}\overline{e}_{i}(\varphi_{V_{1}V_{1}})-\tilde{g}^{i\overline{i}}[e_{i},\overline{e}_{i}]^{(0,1)}(\varphi_{V_{1}V_{1}})\\
= &~~ \tilde{g}^{i\overline{i}}e_{i}\overline{e}_{i}\left(V_{1}V_{1}(\varphi)-(\nabla_{V_{1}}V_{1})\varphi\right)-\tilde{g}^{i\overline{i}}[e_{i},\overline{e}_{i}]^{(0,1)}\left(V_{1}V_{1}(\varphi)-(\nabla_{V_{1}}V_{1})\varphi\right)\\
\geq &~~ \tilde{g}^{i\overline{i}}V_{1}V_{1}\left(e_{i}\overline{e}_{i}(\varphi)-[e_{i},\overline{e}_{i}]^{(0,1)}(\varphi)\right)-2\tilde{g}^{i\overline{i}}[V_{1},e_{i}]V_{1}\overline{e}_{i}(\varphi)\\
&~~ -2\tilde{g}^{i\overline{i}}[V_{1},\overline{e}_{i}]V_{1}e_{i}(\varphi)-\tilde{g}^{i\overline{i}}(\nabla_{V_{1}}V_{1})e_{i}\overline{e}_{i}(\varphi)
+\tilde{g}^{i\overline{i}}(\nabla_{V_{1}}V_{1})[e_{i},\overline{e}_{i}]^{(0,1)}(\varphi)-C\lambda_{1}\sum_{i}\tilde{g}^{i\overline{i}}\\
\geq &~~ \tilde{g}^{i\overline{i}}V_{1}V_{1}(\tilde{g}_{i\overline{i}})-2\tilde{g}^{i\overline{i}}[V_{1},e_{i}]V_{1}\overline{e}_{i}(\varphi)-2\tilde{g}^{i\overline{i}}[V_{1},\overline{e}_{i}]V_{1}e_{i}(\varphi)-\tilde{g}^{i\overline{i}}(\nabla_{V_{1}}V_{1})(\tilde{g}_{i\overline{i}})-C\lambda_{1}\sum_{i}\tilde{g}^{i\overline{i}}.
\end{split}
\end{equation}
Applying $\nabla_{V_{1}}V_{1}$ to (\ref{Parabolic Monge-Ampere Equation}), we obtain
\begin{equation}\label{second order computation equation 4}
\tilde{g}^{i\overline{i}}(\nabla_{V_{1}}V_{1})(\tilde{g}_{i\overline{i}})=(\nabla_{V_{1}}V_{1})F+(\nabla_{V_{1}}V_{1})\frac{\partial\varphi}{\partial t}.
\end{equation}
Similarly, applying $V_{1}$ to (\ref{Parabolic Monge-Ampere Equation}) twice, we get
\begin{equation}\label{second order computation equation 5}
\tilde{g}^{i\overline{i}}V_{1}V_{1}(\tilde{g}_{i\overline{i}})=\tilde{g}^{p\overline{p}}\tilde{g}^{q\overline{q}}|V_{1}(\tilde{g}_{p\overline{q}})|^{2}
+V_{1}V_{1}(F)+V_{1}V_{1}\left(\frac{\partial\varphi}{\partial t}\right).
\end{equation}
By (\ref{Parabolic Monge-Ampere Equation}), Lemma \ref{varphi t estimate} and arithmetic-geometry mean inequality, we have $\sum_{i}\tilde{g}^{i\overline{i}}\geq C^{-1}$ for a uniform constant $C$. It then follows that $(\nabla_{V_{1}}V_{1})F$ and $V_{1}V_{1}(F)$ can be bounded by $C\lambda_{1}\sum_{i}\tilde{g}^{i\overline{i}}$. Combining (\ref{second order computation equation 2}), (\ref{second order computation equation 3}), (\ref{second order computation equation 4}) and (\ref{second order computation equation 5}), we prove (\ref{second order computation equation 1}).
\end{proof}

\begin{lemma}\label{zero and first order computation}
At $(x_{0},t_{0})$, we have
\begin{equation}\label{zero and first order computation equation 1}
\left(L-\frac{\partial}{\partial t}\right)|\partial\varphi|_{g}^{2}\geq \frac{1}{2}\sum_{k}\tilde{g}^{i\overline{i}}\left(|e_{i}e_{k}(\varphi)|^{2}+|e_{i}\overline{e}_{k}(\varphi)|^{2}\right)-C\sum_{i}\tilde{g}^{i\overline{i}}-C
\end{equation}
and
\begin{equation}\label{zero and first order computation equation 2}
\left(L-\frac{\partial}{\partial t}\right)\tilde{\varphi}=n-\sum_{i}\tilde{g}^{i\overline{i}}-\frac{\partial\varphi}{\partial t}+\int_{M}\frac{\partial\varphi}{\partial t}~\omega^{n}.
\end{equation}
\end{lemma}

\begin{proof}
By the same calculation in Section 4 (see (\ref{First order estimate equation 7})), for any $\epsilon\in(0,\frac{1}{2}]$, we have
\begin{equation*}
\begin{split}
\left(L-\frac{\partial}{\partial t}\right)|\partial\varphi|_{g}^{2}
\geq &~~ (1-\epsilon)\sum_{k}\tilde{g}^{i\overline{i}}\left(|e_{i}e_{k}(\varphi)|^{2}+|e_{i}\overline{e}_{k}(\varphi)|^{2}\right)\\
&~~ -\frac{C}{\epsilon}|\partial\varphi|_{g}^{2}\sum_{i}\tilde{g}^{i\overline{i}}+2\text{Re}\left(\sum_{k}\varphi_{k}F_{\overline{k}}\right).
\end{split}
\end{equation*}
By taking $\epsilon=\frac{1}{2}$ and Proposition \ref{First order estimate}, we get (\ref{zero and first order computation equation 1}). For (\ref{zero and first order computation equation 2}), we compute
\begin{equation*}
\left(L-\frac{\partial}{\partial t}\right)\tilde{\varphi}=\tilde{g}^{i\overline{i}}(\tilde{g}_{i\overline{i}}-g_{i\overline{i}})-\frac{\partial\tilde{\varphi}}{\partial t}
=n-\sum_{i}\tilde{g}^{i\overline{i}}-\frac{\partial\varphi}{\partial t}+\int_{M}\frac{\partial\varphi}{\partial t}~\omega^{n},
\end{equation*}
as required.
\end{proof}

For convenience, we use $\sup\tilde{\varphi}$ to denote $\sup_{M\times[0,T)}\tilde{\varphi}$ in the following argument.

\begin{lemma}\label{perturbed quantity computation}
At $(x_{0},t_{0})$, for any $\epsilon\in(0,\frac{1}{2}]$, we have
\begin{equation*}
\begin{split}
0\geq &~~(2-\epsilon)\sum_{\alpha>1}\frac{\tilde{g}^{i\overline{i}}|e_{i}(\varphi_{V_{1}V_{\alpha}})|^{2}}{\lambda_{1}(\lambda_{1}-\lambda_{\alpha})}
+\frac{\tilde{g}^{p\overline{p}}\tilde{g}^{q\overline{q}}|V_{1}(\tilde{g}_{p\overline{q}})|^{2}}{\lambda_{1}}\\
&~~ -(1+\epsilon)\frac{\tilde{g}^{i\overline{i}}|e_{i}(\varphi_{V_{1}V_{1}})|^{2}}{\lambda_{1}^{2}}+\frac{h'}{2}\sum_{k}\tilde{g}^{i\overline{i}}\left(|e_{i}e_{k}(\varphi)|^{2}+|e_{i}\overline{e}_{k}(\varphi)|^{2}\right)\\
&~~
+h''\tilde{g}^{i\overline{i}}|e_{i}(|\partial\varphi|_{g}^{2})|^{2}+\left(Ae^{-A(\tilde{\varphi}-\sup\tilde{\varphi})}-\frac{C}{\epsilon}\right)\sum_{i}\tilde{g}^{i\overline{i}}\\
&~~
+A^{2}e^{-A(\tilde{\varphi}-\sup\tilde{\varphi})}|\partial\varphi|_{\tilde{g}}^{2}-ACe^{-A(\tilde{\varphi}-\sup\tilde{\varphi})}.
\end{split}
\end{equation*}
\end{lemma}

\begin{proof}
By the definitions of vector fields $V_{\alpha}$, the components of $V_{\alpha}$ are constant. Hence, at $(x_{0},t_{0})$, we have
\begin{equation*}
\begin{split}
&~~ |[V_{1},e_{i}]V_{1}\overline{e}_{i}(\varphi)|+|[V_{1},\overline{e}_{i}]V_{1}e_{i}(\varphi)|\\
\leq &~~  C\sum_{\alpha=1}^{2n}|V_{\alpha}V_{1}e_{i}(\varphi)|\\
= &~~ C\sum_{\alpha=1}^{2n}|e_{i}V_{\alpha}V_{1}(\varphi)-V_{\alpha}[e_{i},V_{1}](\varphi)-[e_{i},V_{\alpha}]V_{1}(\varphi)|\\
= &~~ C\sum_{\alpha=1}^{2n}|e_{i}(\varphi_{V_{1}V_{\alpha}})+e_{i}(\nabla_{V_{\alpha}}V_{1})(\varphi)-V_{\alpha}[e_{i},V_{1}](\varphi)-[e_{i},V_{\alpha}]V_{1}(\varphi)|,
\end{split}
\end{equation*}
which implies
\begin{equation}\label{perturbed quantity computation equation 1}
\begin{split}
&~~ 2\tilde{g}^{i\overline{i}}[V_{1},e_{i}]V_{1}\overline{e}_{i}(\varphi)+2\tilde{g}^{i\overline{i}}[V_{1},\overline{e}_{i}]V_{1}e_{i}(\varphi)\\
\leq &~~ C\tilde{g}^{i\overline{i}}|e_{i}(\varphi_{V_{1}V_{1}})|+C\sum_{\alpha>1}\tilde{g}^{i\overline{i}}|e_{i}(\varphi_{V_{1}V_{\alpha}})|+C\lambda_{1}\sum_{i}\tilde{g}^{i\overline{i}}\\
\leq &~~
\epsilon\frac{\tilde{g}^{i\overline{i}}|e_{i}(\varphi_{V_{1}V_{1}})|^{2}}{\lambda_{1}}+\left(\frac{C}{\epsilon}+C\right)\lambda_{1}\sum_{i}\tilde{g}^{i\overline{i}}
+\epsilon\sum_{\alpha>1}\frac{\tilde{g}^{i\overline{i}}|e_{i}(\varphi_{V_{1}V_{\alpha}})|^{2}}{\lambda_{1}-\lambda_{\alpha}}
+\frac{C}{\epsilon}\sum_{i}\sum_{\alpha>1}\tilde{g}^{i\overline{i}}(\lambda_{1}-\lambda_{\alpha})\\
\leq &~~
\epsilon\frac{\tilde{g}^{i\overline{i}}|e_{i}(\varphi_{V_{1}V_{1}})|^{2}}{\lambda_{1}}+\epsilon\sum_{\alpha>1}\frac{\tilde{g}^{i\overline{i}}|e_{i}(\varphi_{V_{1}V_{\alpha}})|^{2}}{\lambda_{1}-\lambda_{\alpha}}
+\frac{C\lambda_{1}}{\epsilon}\sum_{i}\tilde{g}^{i\overline{i}},
\end{split}
\end{equation}
where we used (\ref{lambda 1 control equation 2}) in the last line. Combining the maximum principle, Lemma \ref{second order computation}, Lemma \ref{zero and first order computation}, (\ref{Properties of h}) and (\ref{perturbed quantity computation equation 1}), at $(x_{0},t_{0})$, for any $\epsilon\in(0,\frac{1}{2}]$, we obtain
\begin{equation}\label{perturbed quantity computation equation 2}
\begin{split}
0\geq &~~ \left(L-\frac{\partial}{\partial t}\right)\hat{Q}\\
= &~~
\frac{1}{\lambda_{1}}\left(L-\frac{\partial}{\partial t}\right)\lambda_{1}-\frac{\tilde{g}^{i\overline{i}}|e_{i}(\lambda_{1})|^{2}}{\lambda_{1}^{2}}+h'\left(L-\frac{\partial}{\partial t}\right)|\partial\varphi|_{g}^{2}\\
&~~
+h''\tilde{g}^{i\overline{i}}|e_{i}(|\partial\varphi|_{g}^{2})|^{2}-Ae^{-A(\tilde{\varphi}-\sup\tilde{\varphi})}\left(L-\frac{\partial}{\partial t}\right)\tilde{\varphi}+A^{2}e^{-A(\tilde{\varphi}-\sup\tilde{\varphi})}|\partial\tilde{\varphi}|_{\tilde{g}}^{2}\\
\geq &~~
2\sum_{\alpha>1}\frac{\tilde{g}^{i\overline{i}}|e_{i}(\varphi_{V_{1}V_{\alpha}})|^{2}}{\lambda_{1}(\lambda_{1}-\lambda_{\alpha})}
+\frac{\tilde{g}^{p\overline{p}}\tilde{g}^{q\overline{q}}|V_{1}(\tilde{g}_{p\overline{q}})|^{2}}{\lambda_{1}}\\
&~~ -\frac{\tilde{g}^{i\overline{i}}|e_{i}(\varphi_{V_{1}V_{1}})|^{2}}{\lambda_{1}^{2}}-\frac{2}{\lambda_{1}}\left(\tilde{g}^{i\overline{i}}[V_{1},e_{i}]V_{1}\overline{e}_{i}(\varphi)+\tilde{g}^{i\overline{i}}[V_{1},\overline{e}_{i}]V_{1}e_{i}(\varphi)\right)\\
&~~
+\frac{h'}{2}\sum_{k}\tilde{g}^{i\overline{i}}\left(|e_{i}e_{k}(\varphi)|^{2}+|e_{i}\overline{e}_{k}(\varphi)|^{2}\right)-Ch'\\
&~~
+h''\tilde{g}^{i\overline{i}}|e_{i}(|\partial\varphi|_{g}^{2})|^{2}+\left(Ae^{-A(\tilde{\varphi}-\sup\tilde{\varphi})}-C-Ch'\right)\sum_{i}\tilde{g}^{i\overline{i}}\\
&~~
+A^{2}e^{-A(\tilde{\varphi}-\sup\tilde{\varphi})}|\partial\varphi|_{\tilde{g}}^{2}-A\left(n-\frac{\partial\varphi}{\partial t}+\int_{M}\frac{\partial\varphi}{\partial t}~\omega^{n}\right)e^{-A(\tilde{\varphi}-\sup\tilde{\varphi})},
\end{split}
\end{equation}
where we used the first derivative formula of $\lambda_{1}$ (see \cite[Lemma 5.2]{CTW}) and $|\partial\tilde{\varphi}|_{\tilde{g}}^{2}=|\partial\varphi|_{\tilde{g}}^{2}$ for the last inequality. Therefore, Lemma \ref{perturbed quantity computation} follows from Lemma \ref{varphi t estimate}, (\ref{Properties of h}), (\ref{perturbed quantity computation equation 1}), (\ref{perturbed quantity computation equation 2}) and the fact that $\sum_{i}\tilde{g}^{i\overline{i}}$ has a positive uniform lower bound.
\end{proof}

In what follows, for convenience, we use $C_{A}$ to denote the constant depending only on $(M,\omega,J)$, $F$, $\varphi_{0}$ and $A$. In order to prove Proposition \ref{Second order estimate}, we split up into different cases.

\bigskip
\noindent
{\bf Case 1.} At $(x_{0},t_{0})$, we assume that
\begin{equation}\label{case 1 equation 1}
\tilde{g}_{1\overline{1}}<A^{3}e^{-2A(\tilde{\varphi}-\sup\tilde{\varphi})}\tilde{g}_{n\overline{n}}.
\end{equation}
\bigskip

Since $(x_{0},t_{0})$ is the local maximum point of $\hat{Q}$, thus we have $e_{i}(\hat{Q})=0$ at $(x_{0},t_{0})$, which implies
\begin{equation}\label{case 1 equation 2}
\begin{split}
\frac{\tilde{g}^{i\overline{i}}|e_{i}(\varphi_{V_{1}V_{1}})|^{2}}{\lambda_{1}^{2}}
& = \tilde{g}^{i\overline{i}}|Ae^{-A(\tilde{\varphi}-\sup\tilde{\varphi})}e_{i}(\varphi)-h'e_{i}(|\partial\varphi|_{g}^{2})|^{2}\\
& \leq 4\left(\sup_{M\times[0,T)}|\partial\varphi|_{g}^{2}\right)A^{2}e^{-2A(\tilde{\varphi}-\sup\tilde{\varphi})}\sum_{i}\tilde{g}^{i\overline{i}}
+\frac{4}{3}(h')^{2}\tilde{g}^{i\overline{i}}\left|e_{i}(|\partial\varphi|_{g}^{2})\right|^{2}.
\end{split}
\end{equation}
Combining Lemma \ref{perturbed quantity computation} (taking $\epsilon=\frac{1}{2}$) and (\ref{case 1 equation 2}), we obtain
\begin{equation}\label{case 1 equation 3}
\begin{split}
0 \geq &~~ -6\left(\sup_{M\times[0,T)}|\partial\varphi|_{g}^{2}\right)A^{2}e^{-2A(\tilde{\varphi}-\sup\tilde{\varphi})}\sum_{i}\tilde{g}^{i\overline{i}}+\left(h''-2(h')^{2}\right)\tilde{g}^{i\overline{i}}|e_{i}(|\partial\varphi|_{g}^{2})|^{2}\\
&~~ +\frac{h'}{2}\sum_{k}\tilde{g}^{i\overline{i}}\left(|e_{i}e_{k}(\varphi)|^{2}
+|e_{i}\overline{e}_{k}(\varphi)|^{2}\right)+(Ae^{-A(\tilde{\varphi}-\sup\tilde{\varphi})}-C)\sum_{i}\tilde{g}^{i\overline{i}}-ACe^{-A(\tilde{\varphi}-\sup\tilde{\varphi})}.
\end{split}
\end{equation}
Using (\ref{Parabolic Monge-Ampere Equation}), (\ref{case 1 equation 1}) and Proposition \ref{Oscillation estimate}, at $(x_{0},t_{0})$, it is clear that
\begin{equation}\label{case 1 equation 4}
C_{A}\geq\tilde{g}_{1\overline{1}}\geq\tilde{g}_{2\overline{2}}\geq\cdots\geq\tilde{g}_{n\overline{n}}\geq C_{A}^{-1}.
\end{equation}
Plugging (\ref{Properties of h}) and (\ref{case 1 equation 4}) into (\ref{case 1 equation 3}), we obtain
\begin{equation*}
\sum_{i,k}\left(|e_{i}e_{k}(\varphi)|^{2}+|e_{i}\overline{e}_{k}(\varphi)|^{2}\right)\leq C_{A},
\end{equation*}
which implies the upper bound of $\lambda_{1}$ at $(x_{0},t_{0})$. This completes Case 1.

\bigskip
\noindent
{\bf Case 2.} At $(x_{0},t_{0})$, we assume that
\begin{equation}\label{case 2 equation 1}
\frac{h'}{4}\sum_{k}\tilde{g}^{i\overline{i}}\left(|e_{i}e_{k}(\varphi)|^{2}+|e_{i}\overline{e}_{k}(\varphi)|^{2}\right)
>6\left(\sup_{M\times[0,T)}|\partial\varphi|_{g}^{2}\right)A^{2}e^{-2A(\tilde{\varphi}-\sup\tilde{\varphi})}\sum_{i}\tilde{g}^{i\overline{i}}.
\end{equation}
\bigskip

By the same argument in Case 1, we still have (\ref{case 1 equation 3}). Combining (\ref{Properties of h}), (\ref{case 1 equation 3}), (\ref{case 2 equation 1}) and Proposition \ref{Oscillation estimate}, at $(x_{0},t_{0})$, we get
\begin{equation*}
0\geq C_{1}^{-1}\sum_{k}\tilde{g}^{i\overline{i}}\left(|e_{i}e_{k}(\varphi)|^{2}+|e_{i}\overline{e}_{k}(\varphi)|^{2}\right)
+(A-C_{1})\sum_{i}\tilde{g}^{i\overline{i}}-C_{1}Ae^{C_{1}A},
\end{equation*}
for a uniform constant $C_{1}$. We can choose $A$ sufficiently large such that $A>C_{1}+1$, then we can get the upper bound of $\sum_{i}\tilde{g}^{i\overline{i}}$. Thus, by (\ref{Parabolic Monge-Ampere Equation}), we obtain the positive lower and upper bound of $\tilde{g}_{i\overline{i}}$ for every $i=1,2,\cdots,n$. Hence, by the similar argument in Case 1, we get the upper bound of $\lambda_{1}$ at $(x_{0},t_{0})$, which completes Case 2.

\bigskip
\noindent
{\bf Case 3.} Neither Case 1 nor Case 2 happen.
\bigskip

In this case, we need to deal with the following terms
\begin{equation*}
(2-\epsilon)\sum_{\alpha>1}\frac{\tilde{g}^{i\overline{i}}|e_{i}(\varphi_{V_{1}V_{\alpha}})|^{2}}{\lambda_{1}(\lambda_{1}-\lambda_{\alpha})}
+\frac{\tilde{g}^{p\overline{p}}\tilde{g}^{q\overline{q}}|V_{1}(\tilde{g}_{p\overline{q}})|^{2}}{\lambda_{1}}
-(1+\epsilon)\frac{\tilde{g}^{i\overline{i}}|e_{i}(\varphi_{V_{1}V_{1}})|^{2}}{\lambda_{1}^{2}}
\end{equation*}
in Lemma \ref{perturbed quantity computation}. In order to do this, we define
\begin{equation*}
I=\{i~|~\tilde{g}_{i\overline{i}}\geq A^{3}e^{-2A(\tilde{\varphi}-\sup\tilde{\varphi})}\tilde{g}_{n\overline{n}} \text{~~at $(x_{0},t_{0})$}\}.
\end{equation*}
Since Case 1 does not happen and $A>1$, we have $1\in I$ and $n\notin I$. Without loss of generality, we can assume $I=\{1,2,\cdots,j\}$. By the similar argument of \cite[Lemma 5.5]{CTW}, we obtain

\begin{lemma}\label{case 3 lemma 1}
At $(x_{0},t_{0})$, for any $\epsilon\in(0,\frac{1}{2}]$, we have
\begin{equation*}
-(1+\epsilon)\sum_{i\in I}\frac{\tilde{g}^{i\overline{i}}|e_{i}(\varphi_{V_{1}V_{1}})|^{2}}{\lambda_{1}^{2}}
\geq -\sum_{i}\tilde{g}^{i\overline{i}}-2(h')^{2}\sum_{i\in I}\tilde{g}^{i\overline{i}}\left|e_{i}(|\partial\varphi|_{g}^{2})\right|^{2}.
\end{equation*}
\end{lemma}

Without loss of generality, we can assume $\lambda_{1}\geq\frac{C_{A}}{\epsilon^{3}}$ at $(x_{0},t_{0})$ ($A$ and $\epsilon$ will be chosen uniformly at last). Using the similar arguments of \cite[Lemma 5.6, Lemma 5.7, Lemma 5.8]{CTW}, we have the following estimate

\begin{lemma}\label{case 3 lemma 2}
At $(x_{0},t_{0})$, for any $\epsilon\in(0,\frac{1}{6})$, we have
\begin{equation*}
\begin{split}
&~~(2-\epsilon)\sum_{\alpha>1}\frac{\tilde{g}^{i\overline{i}}|e_{i}(\varphi_{V_{1}V_{\alpha}})|^{2}}{\lambda_{1}(\lambda_{1}-\lambda_{\alpha})}
+\frac{\tilde{g}^{p\overline{p}}\tilde{g}^{q\overline{q}}|V_{1}(\tilde{g}_{p\overline{q}})|^{2}}{\lambda_{1}}
-(1+\epsilon)\sum_{i\notin I}\frac{\tilde{g}^{i\overline{i}}|e_{i}(\varphi_{V_{1}V_{1}})|^{2}}{\lambda_{1}^{2}}\\
\geq &~~ -3\epsilon\sum_{i\notin I}\frac{\tilde{g}^{i\overline{i}}|e_{i}(\varphi_{V_{1}V_{1}})|^{2}}{\lambda_{1}^{2}}-\frac{C}{\epsilon}\sum_{i}\tilde{g}^{i\overline{i}}.
\end{split}
\end{equation*}
\end{lemma}

Now, we prove Proposition \ref{Second order estimate} for Case 3. Combining $\partial\hat{Q}=0$ at $(x_{0},t_{0})$ and the Cauchy-Schwarz inequality, for any $\epsilon\in(0,\frac{1}{6})$, we have
\begin{equation}\label{case 3 equation 1}
\begin{split}
-3\epsilon\sum_{i\notin I}\frac{\tilde{g}^{i\overline{i}}|e_{i}(\varphi_{V_{1}V_{1}})|^{2}}{\lambda_{1}^{2}}
= &~~ -3\epsilon\sum_{i\notin I}\tilde{g}^{i\overline{i}}|Ae^{-A(\tilde{\varphi}-\sup\tilde{\varphi})}e_{i}(\varphi)-h'e_{i}(|\partial\varphi|_{g}^{2})|^{2}\\
\geq &~~ -6\epsilon A^{2}e^{-2A(\tilde{\varphi}-\sup\tilde{\varphi})}|\partial\varphi|_{\tilde{g}}^{2}
-6\epsilon(h')^{2}\sum_{i\notin I}\tilde{g}^{i\overline{i}}|e_{i}(|\partial\varphi|_{g}^{2})|^{2}\\
\geq &~~ -6\epsilon A^{2}e^{-2A(\tilde{\varphi}-\sup\tilde{\varphi})}|\partial\varphi|_{\tilde{g}}^{2}
-2(h')^{2}\sum_{i\notin I}\tilde{g}^{i\overline{i}}|e_{i}(|\partial\varphi|_{g}^{2})|^{2}
\end{split}
\end{equation}
Combining (\ref{Properties of h}), (\ref{case 3 equation 1}), Lemma \ref{perturbed quantity computation}, Lemma \ref{case 3 lemma 1} and Lemma \ref{case 3 lemma 2}, we obtain
\begin{equation}\label{case 3 equation 2}
\begin{split}
0 \geq &~~ \left(Ae^{-A(\tilde{\varphi}-\sup\tilde{\varphi})}-\frac{C_{2}}{\epsilon}\right)\sum_{i}\tilde{g}^{i\overline{i}}
+\frac{h'}{2}\sum_{k}\tilde{g}^{i\overline{i}}\left(|e_{i}e_{k}(\varphi)|^{2}+|e_{i}\overline{e}_{k}(\varphi)|^{2}\right)\\
&~~ +\left(A^{2}e^{-A(\tilde{\varphi}-\sup\tilde{\varphi})}-6\epsilon A^{2}e^{-2A(\tilde{\varphi}-\sup\tilde{\varphi})}\right)|\partial\varphi|_{\tilde{g}}^{2}-AC_{2}e^{-A(\tilde{\varphi}-\sup\tilde{\varphi})},
\end{split}
\end{equation}
where $C_{2}$ is a uniform constant. Now, we choose
\begin{equation*}
A=6C_{2}+1 \text{~~and~~} \epsilon=\frac{1}{6}e^{A(\tilde{\varphi}(x_{0},t_{0})-\sup\tilde{\varphi})}\in(0,\frac{1}{6}).
\end{equation*}
Thus, at $(x_{0},t_{0})$, by Proposition \ref{Oscillation estimate}, (\ref{Properties of h}) and (\ref{case 3 equation 2}), it is clear that
\begin{equation*}
\sum_{i}\tilde{g}^{i\overline{i}}+\sum_{k}\tilde{g}^{i\overline{i}}\left(|e_{i}e_{k}(\varphi)|^{2}
+|e_{i}\overline{e}_{k}(\varphi)|^{2}\right)\leq C,
\end{equation*}
for a uniform constant $C$. By the similar argument in Case 1, we get the upper bound of $\lambda_{1}$ at $(x_{0},t_{0})$, which completes Case 3. Hence, we complete the proof of Proposition \ref{Second order estimate}.
\end{proof}

\section{Proof of (1) in Theorem \ref{Main Theorem}}
In this section, we give the proof of (1) in Theorem \ref{Main Theorem}. First, we need the following estimate.

\begin{lemma}\label{Proof of Main Theorem lemma 1}
Let $\varphi$ be the solution of (\ref{Parabolic Monge-Ampere Equation}) and $[0,T)$ be the maximal time interval. For any $\epsilon\in(0,T)$ and positive integer $k\geq1$, there exists a constant $C(\epsilon,k)$ depending only on $(M,\omega,J)$, $F$, $\varphi_{0}$, $\epsilon$ and $k$ such that
\begin{equation}\label{Proof of Main Theorem lemma 1 equation 1}
\sup_{M\times[\epsilon,T)}|\nabla^{k}\varphi(x,t)|\leq C(\epsilon,k).
\end{equation}
\end{lemma}

\begin{proof}
Combining Proposition \ref{First order estimate} and Proposition \ref{Second order estimate}, it is clear that (\ref{Parabolic Monge-Ampere Equation}) is uniformly parabolic. By the Schauder estimate (see e.g. \cite{Lieberman}) and bootstrapping method, in order to prove (\ref{Proof of Main Theorem lemma 1 equation 1}), it suffices to prove the H\"{o}lder estimate of $\sqrt{-1}\partial\overline{\partial}\varphi$. We split up into different cases.

\bigskip
\noindent
{\bf Case 1.} $T<1$.
\bigskip

In this case, by Lemma \ref{varphi t estimate}, we have
\begin{equation*}
\sup_{M\times[0,T)}|\varphi(x,t)|\leq CT+C \leq C,
\end{equation*}
for a uniform constant $C$. By Theorem 5.1 in \cite{Chu16}, we obtain (\ref{Proof of Main Theorem lemma 1 equation 1}).

\bigskip
\noindent
{\bf Case 2.} $T\geq1$.
\bigskip

In this case, for any $b\in(0,T-1)$, we define
\begin{equation*}
\varphi_{b}(x,t)=\varphi(x,t+b)-\inf_{M\times[b,b+1)}\varphi(x,t)
\end{equation*}
for all $t\in[0,1)$. Combining Lemma \ref{varphi t estimate} and Proposition \ref{Oscillation estimate}, we have
\begin{equation*}
\sup_{M\times[0,1)}|\varphi_{b}(x,t)|\leq C(b+1-1)+C \leq C,
\end{equation*}
for a uniform constant $C$. It is clear that
\begin{equation*}
\frac{\partial\varphi_{b}}{\partial t}=\log\frac{(\omega+\sqrt{-1}\partial\overline{\partial}\varphi_{b})^{n}}{\omega^{n}}-F
\end{equation*}
for any $(x,t)\in M\times[0,1)$. By Theorem 5.1 in \cite{Chu16}, for any $\epsilon\in(0,\frac{1}{2})$ and $\alpha\in(0,1)$, we have
\begin{equation*}
[\sqrt{-1}\partial\overline{\partial}\varphi]_{C^{\alpha}(M\times[b+\epsilon,b+1))}
=[\sqrt{-1}\partial\overline{\partial}\varphi_{b}]_{C^{\alpha}(M\times[\epsilon,1))}
\leq C(\epsilon,\alpha),
\end{equation*}
where $C(\epsilon,\alpha)$ is a constant depending only on $(M,\omega,J)$, $F$, $\varphi_{0}$, $\epsilon$ and $\alpha$. Since $b\in(0,T-1)$ is arbitrary, it is clear that
\begin{equation*}
[\sqrt{-1}\partial\overline{\partial}\varphi]_{C^{\alpha}(M\times[\epsilon,T))}\leq C(\epsilon,\alpha),
\end{equation*}
as required.
\end{proof}

Now, we are in the position to prove (1) in Theorem \ref{Main Theorem}.

\begin{proof}[Proof of (1) in Theorem \ref{Main Theorem}]
First, the uniqueness of solution follows from the standard parabolic theory. Next, to prove $T=\infty$, we argue by contradiction. If $T<\infty$, by Lemma \ref{varphi t estimate}, we have
\begin{equation}\label{Proof of Main Theorem (1) equation 1}
\begin{split}
\sup_{M\times[0,T)}|\varphi(x,t)| & \leq \sup_{M}|\varphi_{0}(x)|+T\sup_{M\times[0,T)}\left|\frac{\partial\varphi}{\partial t}(x,t)\right|\\
& \leq C(T+1),
\end{split}
\end{equation}
for a uniform constant $C$. Combining (\ref{Proof of Main Theorem (1) equation 1}), Lemma \ref{Proof of Main Theorem lemma 1} and short time existence, we can extend the solution $\varphi$ to $[0,T_{0})$ ($T_{0}>T$), which is a contradiction.
\end{proof}

\section{The Harnack inequality}
In this section, we consider the following parabolic equation
\begin{equation}\label{Harnack parabolic equation}
\frac{\partial}{\partial t}u=Lu,
\end{equation}
where $L$ is defined by (\ref{Definition of operator L}) and $\varphi$ is the solution of (\ref{Parabolic Monge-Ampere Equation}). Let $u$ be a positive solution of (\ref{Harnack parabolic equation}). We prove the Harnack inequality on almost Hermitian manifold $(M,\omega,J)$ (see \cite{Weinkove04} for the K\"{a}hler case and see \cite{Gill} for the Hermitian case), which is the generalized version of Theorem 2.2 in \cite{LY}. Our argument is similar to \cite{LY}, which is a little different from the arguments given in \cite{Weinkove04} and \cite{Gill}. First, we have the following lemmas.

\begin{lemma}\label{Uniform Estimate}
Let $\varphi$ be the solution of (\ref{Parabolic Monge-Ampere Equation}). For any positive integer $k\geq1$, there exists a constant $C_{k}$ depending only on $(M,\omega,J)$, $F$, $\varphi_{0}$ and $k$ such that
\begin{equation*}
\sup_{M\times[0,\infty)}|\nabla^{k}\varphi(x,t)|\leq C_{k}.
\end{equation*}
\end{lemma}

\begin{proof}
For convenience, we use $C_{k}$ to denote the constant depending only on $(M,\omega,J)$, $F$, $\varphi_{0}$ and $k$. Combining Lemma \ref{Proof of Main Theorem lemma 1} and (1) in Theorem \ref{Main Theorem}, we obtain
\begin{equation}\label{Uniform Estimate equation 1}
\sup_{M\times[1,\infty)}|\nabla^{k}\varphi(x,t)|\leq C_{k}.
\end{equation}
Since $\varphi$ is uniquely determined by $(M,\omega,J)$, $F$ and $\varphi_{0}$, we have
\begin{equation}\label{Uniform Estimate equation 2}
\sup_{M\times[0,1)}|\nabla^{k}\varphi(x,t)|\leq C_{k}.
\end{equation}
Combining (\ref{Uniform Estimate equation 1}) and (\ref{Uniform Estimate equation 2}), we complete the proof.
\end{proof}

\begin{lemma}\label{Harnack Gradient calculation lemma}
Let $u$ be a positive solution of (\ref{Harnack parabolic equation}). For any $\epsilon\in(0,\frac{1}{2})$, $\alpha>1$ and $t>0$, there exists a constant $C$ depending only on $(M,\omega,J)$, $F$ and $\varphi_{0}$ such that
\begin{equation*}
\begin{split}
\left(L-\frac{\partial}{\partial t}\right)G \geq &~~ \frac{(1-\epsilon)t}{n}\left(|\partial f|_{\tilde{g}}^{2}-\frac{\partial f}{\partial t}\right)^{2}-2\text{Re}\left(g^{i\overline{j}}e_{i}(G)\overline{e}_{j}(f)\right)\\
& -\left(|\partial f|_{\tilde{g}}^{2}-\alpha\frac{\partial f}{\partial t}\right)-\frac{C\alpha t}{\epsilon}|\partial f|_{\tilde{g}}^{2}-\frac{C\alpha^{2}t}{\epsilon},
\end{split}
\end{equation*}
where $G=t(|\partial f|_{\tilde{g}}^{2}-\alpha\frac{\partial f}{\partial t})$, $f=\log u$ and $\tilde{g}$ is the corresponding Riemannian metric of $(M,\tilde{\omega},J)$.
\end{lemma}

\begin{proof}
Since $u$ is a positive solution of (\ref{Harnack parabolic equation}) and $f=\log u$, by direct calculation, we have
\begin{equation}\label{Harnack Gradient calculation lemma equation 1}
\left(L-\frac{\partial}{\partial t}\right)f=-|\partial f|_{\tilde{g}}^{2}.
\end{equation}
Let $\{e_{i}\}_{i=1}^{n}$ be a local frame for $T_{\mathbb{C}}^{(1,0)}M$. First, we compute
\begin{equation}\label{Harnack Gradient calculation lemma equation 2}
\begin{split}
L(|\partial f|_{\tilde{g}}^{2}) = &~~ \tilde{g}^{k\overline{l}}e_{k}\overline{e}_{l}\left(\tilde{g}^{i\overline{j}}e_{i}(f)\overline{e}_{j}(f)\right)-\tilde{g}^{k\overline{l}}[e_{k},\overline{e}_{l}]^{(0,1)}\left(\tilde{g}^{i\overline{j}}e_{i}(f)\overline{e}_{j}(f)\right)\\
= &~~ \tilde{g}^{k\overline{l}}e_{k}\overline{e}_{l}(\tilde{g}^{i\overline{j}})e_{i}(f)\overline{e}_{j}(f)+\tilde{g}^{k\overline{l}}
\tilde{g}^{i\overline{j}}e_{k}\overline{e}_{l}e_{i}(f)\overline{e}_{j}(f)\\
&+\tilde{g}^{k\overline{l}}\tilde{g}^{i\overline{j}}e_{i}(f)e_{k}\overline{e}_{l}\overline{e}_{j}(f)+2\text{Re}\left(\tilde{g}^{k\overline{l}}e_{k}(\tilde{g}^{i\overline{j}})\overline{e}_{l}e_{i}(f)\overline{e}_{j}(f)\right)\\
&+2\text{Re}\left(\tilde{g}^{k\overline{l}}\overline{e}_{l}(\tilde{g}^{i\overline{j}})e_{k}e_{i}(f)\overline{e}_{j}(f)\right)+\tilde{g}^{k\overline{l}}\tilde{g}^{i\overline{j}}e_{k}e_{i}(f)\overline{e}_{l}\overline{e}_{j}(f)\\
&+\tilde{g}^{k\overline{l}}\tilde{g}^{i\overline{j}}\overline{e}_{l}e_{i}(f)e_{k}\overline{e}_{j}(f)-\tilde{g}^{k\overline{l}}[e_{k},\overline{e}_{l}]^{(0,1)}(\tilde{g}^{i\overline{j}})e_{i}(f)\overline{e}_{j}(f)\\
&-\tilde{g}^{k\overline{l}}\tilde{g}^{i\overline{j}}[e_{k},\overline{e}_{l}]^{(0,1)}e_{i}(f)\overline{e}_{j}(f)-\tilde{g}^{k\overline{l}}\tilde{g}^{i\overline{j}}e_{i}(f)[e_{k},\overline{e}_{l}]^{(0,1)}\overline{e}_{j}(f).
\end{split}
\end{equation}
For the first and eighth terms of (\ref{Harnack Gradient calculation lemma equation 2}), by Lemma \ref{Uniform Estimate}, we have
\begin{equation*}
\tilde{g}^{k\overline{l}}e_{k}\overline{e}_{l}(\tilde{g}^{i\overline{j}})e_{i}(f)\overline{e}_{j}(f)
-\tilde{g}^{k\overline{l}}[e_{k},\overline{e}_{l}]^{(0,1)}(\tilde{g}^{i\overline{j}})e_{i}(f)\overline{e}_{j}(f)\geq-C|\partial f|_{\tilde{g}}^{2},
\end{equation*}
for a uniform constant $C$. For the fourth and fifth terms of (\ref{Harnack Gradient calculation lemma equation 2}), we get
\begin{equation*}
\begin{split}
~~&2\text{Re}\left(\tilde{g}^{k\overline{l}}e_{k}(\tilde{g}^{i\overline{j}})\overline{e}_{l}e_{i}(f)\overline{e}_{j}(f)\right)+2\text{Re}\left(\tilde{g}^{k\overline{l}}\overline{e}_{l}(\tilde{g}^{i\overline{j}})e_{k}e_{i}(f)\overline{e}_{j}(f)\right)\\
\leq~~&\frac{\epsilon}{30}\tilde{g}^{k\overline{l}}\tilde{g}^{i\overline{j}}\left(\overline{e}_{l}e_{i}(f)e_{k}\overline{e}_{j}(f)+e_{k}e_{i}(f)\overline{e}_{l}\overline{e}_{j}(f)\right)+\frac{C}{\epsilon}|\partial f|_{\tilde{g}}^{2}.
\end{split}
\end{equation*}
For the second and ninth terms of (\ref{Harnack Gradient calculation lemma equation 2}), we obtain
\begin{equation*}
\begin{split}
~~&\tilde{g}^{k\overline{l}}\tilde{g}^{i\overline{j}}e_{k}\overline{e}_{l}e_{i}(f)\overline{e}_{j}(f)-\tilde{g}^{k\overline{l}}\tilde{g}^{i\overline{j}}[e_{k},\overline{e}_{l}]^{(0,1)}e_{i}(f)\overline{e}_{j}(f)\\
\geq~~&\tilde{g}^{i\overline{j}}e_{i}(Lf)\overline{e}_{j}(f)+E\cdot|\partial f|_{\tilde{g}}-C|\partial f|_{\tilde{g}}^{2}\\
\geq~~&\tilde{g}^{i\overline{j}}e_{i}(Lf)\overline{e}_{j}(f)-\frac{\epsilon}{30}\tilde{g}^{k\overline{l}}\tilde{g}^{i\overline{j}}\left(\overline{e}_{l}e_{i}(f)e_{k}\overline{e}_{j}(f)+e_{k}e_{i}(f)\overline{e}_{l}\overline{e}_{j}(f)\right)
-\frac{C}{\epsilon}|\partial f|_{\tilde{g}}^{2},
\end{split}
\end{equation*}
where in the second line the term $E$ just contains second derivatives of $f$. Similarly, for the third and tenth terms of (\ref{Harnack Gradient calculation lemma equation 2}), we have
\begin{equation*}
\begin{split}
~~&\tilde{g}^{k\overline{l}}\tilde{g}^{i\overline{j}}e_{i}(f)e_{k}\overline{e}_{l}\overline{e}_{j}(f)-\tilde{g}^{k\overline{l}}\tilde{g}^{i\overline{j}}e_{i}(f)[e_{k},\overline{e}_{l}]^{(0,1)}\overline{e}_{j}(f)\\
\geq~~&\tilde{g}^{i\overline{j}}e_{i}(f)\overline{e}_{j}(Lf)-\frac{\epsilon}{30}\tilde{g}^{k\overline{l}}g^{i\overline{j}}\left(\overline{e}_{l}e_{i}(f)e_{k}\overline{e}_{j}(f)+e_{k}e_{i}(f)\overline{e}_{l}\overline{e}_{j}(f)\right)
-\frac{C}{\epsilon}|\partial f|_{\tilde{g}}^{2}.
\end{split}
\end{equation*}
Plugging these inequalities into (\ref{Harnack Gradient calculation lemma equation 2}), we obtain
\begin{equation*}
\begin{split}
L(|\partial f|_{\tilde{g}}^{2})\geq &~~ \left(1-\frac{\epsilon}{10}\right)\tilde{g}^{k\overline{l}}\tilde{g}^{i\overline{j}}\left(e_{k}e_{i}(f)\overline{e}_{l}\overline{e}_{j}(f)+\overline{e}_{l}e_{i}(f)e_{k}\overline{e}_{j}(f)\right)
-\frac{C}{\epsilon}|\partial f|_{\tilde{g}}^{2}\\
& +\tilde{g}^{i\overline{j}}e_{i}(Lf)\overline{e}_{j}(f)+\tilde{g}^{i\overline{j}}e_{i}(f)\overline{e}_{j}(Lf),
\end{split}
\end{equation*}
which implies
\begin{equation}\label{Harnack Gradient calculation lemma equation 3}
\begin{split}
\left(L-\frac{\partial}{\partial t}\right)|\partial f|_{\tilde{g}}^{2}\geq &~~ \left(1-\frac{\epsilon}{10}\right)\tilde{g}^{k\overline{l}}\tilde{g}^{i\overline{j}}\left(e_{k}e_{i}(f)\overline{e}_{l}\overline{e}_{j}(f)+\overline{e}_{l}e_{i}(f)e_{k}\overline{e}_{j}(f)\right)
-\frac{C}{\epsilon}|\partial f|_{\tilde{g}}^{2}\\
& -\tilde{g}^{i\overline{j}}e_{i}\left(|\partial f|_{\tilde{g}}^{2}\right)\overline{e}_{j}(f)
-\tilde{g}^{i\overline{j}}e_{i}(f)\overline{e}_{j}\left(|\partial f|_{\tilde{g}}^{2}\right),
\end{split}
\end{equation}
where we used (\ref{Harnack Gradient calculation lemma equation 1}) and Lemma \ref{Uniform Estimate} (note that Lemma \ref{Uniform Estimate} implies $-C\tilde{g}_{i\overline{j}}\leq\frac{\partial\tilde{g}_{i\overline{j}}}{\partial t}\leq C\tilde{g}_{i\overline{j}}$ for a uniform constant $C$). Next, by the Cauchy-Schwarz inequality, we have
\begin{equation}\label{Harnack Gradient calculation lemma equation 4}
\begin{split}
&~~ \left(L-\frac{\partial}{\partial t}\right)\frac{\partial f}{\partial t}\\
= &~~ \tilde{g}^{i\overline{j}}e_{i}\overline{e}_{j}\left(\frac{\partial f}{\partial t}\right)-\tilde{g}^{i\overline{j}}[e_{i},\overline{e}_{j}]^{(0,1)}\left(\frac{\partial f}{\partial t}\right)-\frac{\partial^{2} f}{\partial t^{2}}\\
= &~~ \frac{\partial}{\partial t}\left(L-\frac{\partial}{\partial t}\right)f-\frac{\partial\tilde{g}^{i\overline{j}}}{\partial t}e_{i}\overline{e}_{j}(f)+\frac{\partial\tilde{g}^{i\overline{j}}}{\partial t}[e_{i},\overline{e}_{j}]^{(0,1)}(f)\\
\leq & -\frac{\partial}{\partial t}|\partial f|_{\tilde{g}}^{2}+\frac{\epsilon}{10\alpha}\tilde{g}^{k\overline{l}}\tilde{g}^{i\overline{j}}\overline{e}_{l}e_{i}(f)e_{k}\overline{e}_{j}(f)+C|\partial f|_{\tilde{g}}^{2}+\frac{C\alpha}{\epsilon}\\
\leq & -2\text{Re}\left(\tilde{g}^{i\overline{j}}e_{i}(f_{t})\overline{e}_{j}(f)\right)+\frac{\epsilon}{10\alpha}\tilde{g}^{k\overline{l}}\tilde{g}^{i\overline{j}}\overline{e}_{l}e_{i}(f)e_{k}\overline{e}_{j}(f)
+C|\partial f|_{\tilde{g}}^{2}+\frac{C\alpha}{\epsilon},
\end{split}
\end{equation}
where we used (\ref{Harnack Gradient calculation lemma equation 1}) and the fact that $-C\tilde{g}_{i\overline{j}}\leq\frac{\partial\tilde{g}_{i\overline{j}}}{\partial t}\leq C\tilde{g}_{i\overline{j}}$ for a uniform constant $C$. By the arithmetic-geometric mean inequality and the Cauchy-Schwarz inequality, we obtain
\begin{equation}\label{Harnack Gradient calculation lemma equation 5}
\begin{split}
\tilde{g}^{k\overline{l}}\tilde{g}^{i\overline{j}}\overline{e}_{l}e_{i}(f)e_{k}\overline{e}_{j}(f) \geq &~~ \frac{1}{n}\left(\tilde{g}^{i\overline{j}}e_{i}\overline{e}_{j}(f)\right)^{2}\\
\geq & \left(1-\frac{\epsilon}{5}\right)\frac{(Lf)^{2}}{n}-\frac{C}{\epsilon}|\partial f|_{\tilde{g}}^{2}.
\end{split}
\end{equation}
Combining (\ref{Harnack Gradient calculation lemma equation 1}), (\ref{Harnack Gradient calculation lemma equation 3}), (\ref{Harnack Gradient calculation lemma equation 4}) and (\ref{Harnack Gradient calculation lemma equation 5}), we have
\begin{equation*}
\begin{split}
\left(L-\frac{\partial}{\partial t}\right)G
= &~~ t\left(L-\frac{\partial}{\partial t}\right)|\partial f|_{\tilde{g}}^{2}
-\alpha t\left(L-\frac{\partial}{\partial t}\right)\frac{\partial f}{\partial t}
-\left(|\partial f|_{\tilde{g}}^{2}-\alpha\frac{\partial f}{\partial t}\right)\\
\geq &~~ \left(1-\frac{\epsilon}{5}\right)t\tilde{g}^{k\overline{l}}\tilde{g}^{i\overline{j}}\overline{e}_{l}e_{i}(f)e_{k}\overline{e}_{j}(f)
-2\text{Re}\left(\tilde{g}^{i\overline{j}}e_{i}(G)\overline{e}_{j}(f)\right)\\
&~~ -\left(|\partial f|_{\tilde{g}}^{2}-\alpha\frac{\partial f}{\partial t}\right)-\frac{C\alpha t}{\epsilon}|\partial f|_{\tilde{g}}^{2}-\frac{C\alpha^{2}t}{\epsilon}\\
\geq &~~ \frac{(1-\epsilon)t}{n}\left(|\partial f|_{\tilde{g}}^{2}-\frac{\partial f}{\partial t}\right)^{2}-2\text{Re}\left(\tilde{g}^{i\overline{j}}e_{i}(G)\overline{e}_{j}(f)\right)\\
&~~ -\left(|\partial f|_{\tilde{g}}^{2}-\alpha\frac{\partial f}{\partial t}\right)-\frac{C\alpha t}{\epsilon}|\partial f|_{\tilde{g}}^{2}-\frac{C\alpha^{2}t}{\epsilon},
\end{split}
\end{equation*}
as required.
\end{proof}

By using Lemma \ref{Harnack Gradient calculation lemma} and the maximum principle, we have the following estimate.

\begin{lemma}\label{Harnack Gradient estimate lemma}
Let $u$ be a positive solution of (\ref{Harnack parabolic equation}). For any $\epsilon\in(0,\frac{1}{2})$, $\alpha>1$ and $t>0$, there exists a constant $C$ depending only on $(M,\omega,J)$, $F$ and $\varphi_{0}$ such that
\begin{equation*}
|\partial f|_{\tilde{g}}^{2}-\alpha \frac{\partial f}{\partial t}\leq \frac{C\alpha^{3}}{\epsilon(\alpha-1)}+\frac{n\alpha^{2}}{(1-\epsilon)t}.
\end{equation*}
\end{lemma}

\begin{proof}
For any $t'>0$, we assume
\begin{equation*}
\max_{M\times[0,t']}G(x,t)=G(x_{0},t_{0}),
\end{equation*}
where $(x_{0},t_{0})\in M\times[0,t']$. It then follows that $G(x_{0},t_{0})\geq G(x_{0},0)=0$. Without loss of generality, we further assume that $t_{0}>0$. Combining the maximum principle and Lemma \ref{Harnack Gradient calculation lemma}, at $(x_{0},t_{0})$, we have
\begin{equation}\label{Harnack Gradient estimate lemma equation 1}
\frac{(1-\epsilon)t_{0}^{2}}{n}\left(|\partial f|_{\tilde{g}}^{2}-\frac{\partial f}{\partial t}\right)^{2}-G-\frac{C\alpha t_{0}^{2}}{\epsilon}|\partial f|_{\tilde{g}}^{2}-\frac{C\alpha^{2}t_{0}^{2}}{\epsilon} \leq 0,
\end{equation}
for a uniform constant $C$. By direct calculation, at $(x_{0},t_{0})$, we obtain
\begin{equation}\label{Harnack Gradient estimate lemma equation 2}
\begin{split}
t_{0}^{2}\left(|\partial f|_{\tilde{g}}^{2}-\frac{\partial f}{\partial t}\right)^{2}
& = \frac{t_{0}^{2}}{\alpha^{2}}\left(|\partial f|_{\tilde{g}}^{2}-\alpha\frac{\partial f}{\partial t}+(\alpha-1)|\partial f|_{\tilde{g}}^{2}\right)^{2}\\
& = \frac{G^{2}}{\alpha^{2}}+\left(\frac{\alpha-1}{\alpha}\right)^{2}t_{0}^{2}|\partial f|_{\tilde{g}}^{4}
+\frac{2(\alpha-1)Gt_{0}}{\alpha^{2}}|\partial f|_{\tilde{g}}^{2}\\
& \geq \frac{G^{2}}{\alpha^{2}}+\left(\frac{\alpha-1}{\alpha}\right)^{2}t_{0}^{2}|\partial f|_{\tilde{g}}^{4},
\end{split}
\end{equation}
where we used $\alpha>1$ and $G(x_{0},t_{0})\geq 0$. Next, by using the inequality $ax^{2}-bx\geq-\frac{b^{2}}{4a}$ for any $a>0$, $b\geq 0$ and $x\in\mathbb{R}$, at $(x_{0},t_{0})$, we have
\begin{equation}\label{Harnack Gradient estimate lemma equation 3}
\frac{1-\epsilon}{n}\left(\frac{\alpha-1}{\alpha}\right)^{2}|\partial f|_{\tilde{g}}^{4}-\frac{C\alpha}{\epsilon}|\partial f|_{\tilde{g}}^{2}\geq-\frac{C\alpha^{4}}{\epsilon^{2}(\alpha-1)^{2}}.
\end{equation}
Combining (\ref{Harnack Gradient estimate lemma equation 1}), (\ref{Harnack Gradient estimate lemma equation 2}) and (\ref{Harnack Gradient estimate lemma equation 3}), at $(x_{0},t_{0})$, it is clear that
\begin{equation*}
\frac{1-\epsilon}{n\alpha^{2}}G^{2}-G-\frac{C\alpha^{4}t_{0}^{2}}{\epsilon^{2}(\alpha-1)^{2}} \leq 0,
\end{equation*}
which implies
\begin{equation*}
G^{2}-\frac{n\alpha^{2}}{1-\epsilon}G-\frac{Cn\alpha^{6}t_{0}^{2}}{\epsilon^{2}(1-\epsilon)(\alpha-1)^{2}} \leq 0.
\end{equation*}
This is a quadratic inequality of $G$. It then follows that
\begin{equation*}
\begin{split}
G(x_{0},t_{0}) & \leq
\frac{1}{2}\left(\frac{n\alpha^{2}}{1-\epsilon}+\sqrt{\left(\frac{n\alpha^{2}}{1-\epsilon}\right)^{2}+\frac{4Cn\alpha^{6}t_{0}^{2}}{\epsilon^{2}(1-\epsilon)(\alpha-1)^{2}}}\right)\\
& \leq \frac{n\alpha^{2}}{1-\epsilon}+\frac{C\alpha^{3}t_{0}}{\epsilon(\alpha-1)}.
\end{split}
\end{equation*}
By the definition of $(x_{0},t_{0})$, for any point $x\in M$, we have
\begin{equation*}
\begin{split}
G(x,t') & \leq G(x_{0},t_{0})\\
& \leq \frac{n\alpha^{2}}{1-\epsilon}+\frac{C\alpha^{3}t_{0}}{\epsilon(\alpha-1)}\\
& \leq \frac{n\alpha^{2}}{1-\epsilon}+\frac{C\alpha^{3}t'}{\epsilon(\alpha-1)}.
\end{split}
\end{equation*}
By the definition of $G(x,t')$, we have
\begin{equation*}
\left(|\partial f|_{\tilde{g}}^{2}-\alpha \frac{\partial f}{\partial t}\right)(x,t')\leq \frac{C\alpha^{3}}{\epsilon(\alpha-1)}+\frac{n\alpha^{2}}{(1-\epsilon)t'}.
\end{equation*}
Since $(x,t')$ is arbitrary, we complete the proof.
\end{proof}

By using Lemma \ref{Uniform Estimate} and Lemma \ref{Harnack Gradient estimate lemma}, we prove the following Harnack inequality.

\begin{proposition}\label{Harnack inequality}
Let $u$ be a positive solution of (\ref{Harnack parabolic equation}). For any $\epsilon\in(0,\frac{1}{2})$, $\alpha>1$ and $t_{2}>t_{1}>0$, there exists a constant $C$ depending only on $(M,\omega,J)$, $F$ and $\varphi_{0}$ such that
\begin{equation}\label{Harnack inequality equation 1}
\sup_{x\in M}u(x,t_{1})\leq\inf_{x\in M}u(x,t_{2})
\left(\frac{t_{2}}{t_{1}}\right)^{\frac{n\alpha}{1-\epsilon}}\exp\left(\frac{C\alpha}{t_{2}-t_{1}}+\frac{C(t_{2}-t_{1})\alpha^{2}}{\epsilon(\alpha-1)}\right).
\end{equation}
\end{proposition}

\begin{proof}
For any $x,y\in M$, let $\gamma:[0,1]\rightarrow M$ be the minimal geodesic from $y$ to $x$ (with respect to $g$). By Lemma \ref{Uniform Estimate} and Lemma \ref{Harnack Gradient estimate lemma}, we compute
\begin{equation}\label{Harnack inequality equation 2}
\begin{split}
& \log\frac{u(x,t_{1})}{u(y,t_{2})}\\
= & \int_{0}^{1} \frac{d}{ds}f\left(\gamma(s),(1-s)t_{2}+st_{1}\right)ds\\
\leq & \int_{0}^{1}\left(C|\partial f|_{g}-(t_{2}-t_{1})\frac{\partial f}{\partial t}\right)ds\\
\leq & \int_{0}^{1}\left(C|\partial f|_{\tilde{g}}-\frac{t_{2}-t_{1}}{\alpha}|\partial f|_{\tilde{g}}^{2}\right)
+\frac{t_{2}-t_{1}}{\alpha}\left(\frac{C\alpha^{3}}{\epsilon(\alpha-1)}+\frac{n\alpha^{2}}{(1-\epsilon)t}\right)ds\\
\leq &~ \frac{C\alpha}{t_{2}-t_{1}}+\frac{C(t_{2}-t_{1})\alpha^{2}}{\epsilon(\alpha-1)}+\frac{n\alpha}{1-\epsilon}\log\left(\frac{t_{2}}{t_{1}}\right),
\end{split}
\end{equation}
where $t=(1-s)t_{2}+st_{1}$ in the fourth line. By (\ref{Harnack inequality equation 2}), we obtain (\ref{Harnack inequality equation 1}).
\end{proof}

\section{Proof of (2) in Main Theorem}
In this section, we give the proof of (2) in Main Theorem.
\begin{proof}[Proof of (2) in Main Theorem]
We define $u=\frac{\partial\varphi}{\partial t}$. First, we claim that for any $t>0$, there exist constants $C$ and $\eta$ depending only on $(M,\omega,J)$, $F$ and $\varphi_{0}$ such that
\begin{equation}\label{Proof of (2) equation 1}
\theta(t)\leq Ce^{-\eta t},
\end{equation}
where $\theta(t)=\sup_{x\in M} u(x,t)-\inf_{x\in M} u(x,t)$. In order to prove (\ref{Proof of (2) equation 1}), we define
\begin{equation*}
v_{m}(x,t)=\sup_{y\in M}u(y,m-1)-u(x,m-1+t)
\end{equation*}
and
\begin{equation*}
w_{m}(x,t)=u(x,m-1+t)-\inf_{y\in M}u(y,m-1),
\end{equation*}
Without loss of generality, we can assume $u(x,m-1)$ is not constant. By the maximum principle, we obtain that $v_{m}$ and $w_{m}$ two positive solutions of (\ref{Harnack parabolic equation}). By Proposition \ref{Harnack inequality} (taking $\epsilon=\frac{1}{3}$, $\alpha=2$, $t_{1}=\frac{1}{2}$ and $t_{2}=1$), it is clear that
\begin{equation}\label{Proof of (2) equation 2}
\sup_{x\in M}u(x,m-1)-\inf_{x\in M}u(x,m-\frac{1}{2})\leq C\left(\sup_{x\in M}u(x,m-1)-\sup_{x\in M}u(x,m)\right)
\end{equation}
and
\begin{equation}\label{Proof of (2) equation 3}
\sup_{x\in M}u(x,m-\frac{1}{2})-\inf_{x\in M}u(x,m-1)\leq C\left(\inf_{x\in M}u(x,m)-\inf_{x\in M}u(x,m-1)\right).
\end{equation}
Combining (\ref{Proof of (2) equation 2}) and (\ref{Proof of (2) equation 3}), we obtain
\begin{equation*}
\theta(m-1)\leq\theta(m-1)+\theta(m-\frac{1}{2})\leq C\left(\theta(m-1)-\theta(m)\right),
\end{equation*}
which implies
\begin{equation*}
\theta(m)\leq\frac{C-1}{C}\theta(m-1).
\end{equation*}
By induction, we complete the proof of (\ref{Proof of (2) equation 1}).

Next, by the definition of $\tilde{\varphi}$, we get $\int_{M}\tilde{\varphi}~\omega^{n}=0$, which implies $\int_{M}\frac{\partial\tilde{\varphi}}{\partial t}\omega^{n}=0$. Hence, there exists $y\in M$ such that $\frac{\partial\tilde{\varphi}}{\partial t}(y,t)=0$. For any $x\in M$, we have
\begin{equation}\label{Proof of (2) equation 4}
\begin{split}
\left|\frac{\partial\tilde{\varphi}}{\partial t}(x,t)\right| & = \left|\frac{\partial\tilde{\varphi}}{\partial t}(x,t)-\frac{\partial\tilde{\varphi}}{\partial t}(y,t)\right|\\
& = |u(x,t)-u(y,t)|\\
& \leq \theta(t)\\
& \leq Ce^{-\eta t},
\end{split}
\end{equation}
where we used (\ref{Proof of (2) equation 1}) in the last line. By the definition of $\tilde{\varphi}$ and Lemma \ref{Uniform Estimate}, for any positive integer $k\geq 1$, it is clear that
\begin{equation}\label{Proof of (2) equation 6}
\sup_{M\times[0,\infty)}|\nabla^{k}\tilde{\varphi}(x,t)|=\sup_{M\times[0,\infty)}|\nabla^{k}\varphi(x,t)|\leq C_{k},
\end{equation}
where $C_{k}$ is the constant depending only on $(M,\omega,J)$, $F$, $\varphi_{0}$ and $k$. Combining (\ref{Proof of (2) equation 4}), (\ref{Proof of (2) equation 6}) and Arzela-Ascoli Theorem, there exists a smooth function $\tilde{\varphi}_{\infty}$ such that
\begin{equation}\label{Proof of (2) equation 5}
\tilde{\varphi}\overset{C^{\infty}}{\longrightarrow}\tilde{\varphi}_{\infty} \text{~~as $t\rightarrow\infty$}.
\end{equation}
By the definition of $\tilde{\varphi}$, (\ref{Parabolic Monge-Ampere Equation}) can be written as
\begin{equation*}
\frac{\partial\tilde{\varphi}}{\partial t}
=\log\frac{(\omega+\sqrt{-1}\partial\overline{\partial}\tilde{\varphi})^{n}}{\omega^{n}}-F
-\int_{M}\left(\log\frac{(\omega+\sqrt{-1}\partial\overline{\partial}\tilde{\varphi})^{n}}{\omega^{n}}-F\right)\omega^{n}.
\end{equation*}
Let $t\rightarrow\infty$, by (\ref{Proof of (2) equation 4}) and (\ref{Proof of (2) equation 5}), we obtain
\begin{equation*}
(\omega+\sqrt{-1}\partial\overline{\partial}\tilde{\varphi}_{\infty})^{n}=e^{F+b}\omega^{n},
\end{equation*}
where
\begin{equation*}
b=\int_{M}\left(\log\frac{(\omega+\sqrt{-1}\partial\overline{\partial}\tilde{\varphi}_{\infty})^{n}}{\omega^{n}}-F\right)\omega^{n}.
\end{equation*}
The uniqueness of $(\tilde{\varphi}_{\infty},b)$ follows from the maximum principle (see \cite[Section 6]{CTW}).
\end{proof}


\begin{thebibliography}{99}

\bibitem{Calabi} \textit{E. Calabi,} {On K\"{a}hler manifolds with vanishing canonical class,} Algebraic geometry and topology. A symposium in honor of S. Lefschetz, pp. 78--89. Princeton University Press, Princeton, N. J., 1957.

\bibitem{Cao} \textit{H.-D. Cao,} {Deformation of K\"{a}hler metrics to K\"{a}hler-Einstein metrics on compact K\"{a}hler manifolds,} Invent. Math. \textbf{81} (1985), no. 2, 359--372.

\bibitem{Cherrier} \textit{P. Cherrier,} {\'{E}quations de Monge-Amp\`{e}re sur les vari\'{e}t\'{e}s Hermitiennes compactes,} Bull. Sc. Math \textbf{111} (1987), 343--385.

\bibitem{Chu15} \textit{J. Chu,} {The complex Monge-Amp\`{e}re equation on some compact Hermitian manifolds,} Pacific J. Math. \textbf{276} (2015), no. 2, 369--386.

\bibitem{Chu16} \textit{J. Chu,} {$C^{2,\alpha}$ regularities and estimates for nonlinear elliptic and parabolic equations in geometry,} Calc. Var. Partial Differential Equations \textbf{55} (2016), no. 1, Art. 8, 20 pp.

\bibitem{CTW} \textit{J. Chu, V. Tosatti} and \textit{B. Weinkove,} {The Monge-Amp\`ere equation for non-integrable almost complex structures,} preprint, arXiv:1603.00706.

\bibitem{DZZ} \textit{S. Dinew, X. Zhang} and \textit{X. Zhang,} {The $C^{2,\alpha}$ estimate of complex Monge-Amp\`{e}re equation,} Indiana Univ. Math. J. \textbf{60} (2011), no. 5, 1713--1722.

\bibitem{Gill} \textit{M. Gill,} {Convergence of the parabolic complex Monge-Amp\`{e}re equation on compact Hermitian manifolds,} Comm. Anal. Geom. \textbf{19} (2011), no. 2, 277--303.

\bibitem{GL10} \textit{B. Guan} and \textit{Q. Li,} {Complex Monge-Amp\`{e}re equations and totally real submanifolds,} Adv. Math. \textbf{225} (2010), no. 3, 1185--1223.

\bibitem{GL12} \textit{B. Guan} and \textit{Q. Li,} {A Monge-Amp\`{e}re type fully nonlinear equation on Hermitian manifolds,} Discrete Contin. Dyn. Syst. Ser. B \textbf{17} (2012), no. 6, 1991--1999.

\bibitem{GS} \textit{B. Guan} and \textit{W. Sun,} {On a class of fully nonlinear elliptic equations on Hermitian manifolds,} Calc. Var. Partial Differential Equations \textbf{54} (2015), no. 1, 901--916.

\bibitem{Hanani} \textit{A. Hanani,} {\'{E}quations du type de Monge-Amp\`{e}re sur les vari\'{e}t\'{e}s hermitiennes compactes,} J. Funct. Anal. \textbf{137} (1996), no.1, 49--75.

\bibitem{HL} \textit{F. R. Harvey} and \textit{H. B. Lawson,} {Potential theory on almost complex manifolds,} Ann. Inst. Fourier (Grenoble) {\bf 65} (2015), no. 1, 171--210.

\bibitem{Li} \textit{Y. Li,} {A priori estimates for Donaldson's equation over compact Hermitian manifolds,} Calc. Var. Partial Differential Equations \textbf{50} (2014), no. 3-4, 867--882.

\bibitem{LY} \textit{P. Li} and \textit{S.-T. Yau,} {On the parabolic kernel of the Schr\"{o}dinger operator,} Acta Math. \textbf{156} (1986), no. 3-4, 153--201.

\bibitem{Lieberman} \textit{G. M. Lieberman,} {Second order parabolic differential equations,} World Scientific Publishing Co., Inc., River Edge, NJ, 1996. xii+439 pp. ISBN: 981-02-2883-X.

\bibitem{PSS} \textit{D. H. Phong, J. Song} and \textit{J. Sturm,} {Complex Monge-Amp\`{e}re equations,} Surveys in differential geometry. Vol. XVII, 327--410, Surv. Differ. Geom., 17, Int. Press, Boston, MA, 2012.

\bibitem{Plis} \textit{S. Pli\'{s},} {The Monge-Amp\`{e}re equation on almost complex manifolds,} Math. Z. \textbf{276} (2014), no. 3-4, 969--983.

\bibitem{Sun15} \textit{W. Sun,} {Parabolic complex Monge-Amp\`{e}re type equations on closed Hermitian manifolds,} Calc. Var. Partial Differential Equations \textbf{54} (2015), no. 4, 3715--3733.

\bibitem{Sun16} \textit{W. Sun,} {On a Class of Fully Nonlinear Elliptic Equations on Closed Hermitian Manifolds,} J. Geom. Anal. \textbf{26} (2016), no. 3, 2459--2473.

\bibitem{Szekelyhidi} \textit{G. Sz\'ekelyhidi,} {Fully non-linear elliptic equations on compact Hermitian manifolds,} preprint, arxiv:1501.02762.

\bibitem{STW} \textit{G. Sz\'ekelyhidi, V. Tosatti} and \textit{B. Weinkove,} {\em Gauduchon metrics with prescribed volume form}, preprint, arxiv:1503.04491.

\bibitem{Tian83} \textit{G. Tian,} {On the existence of solutions of a class of Monge-Amp\`{e}re equations,} A Chinese summary appears in Acta Math. Sinica \textbf{32} (1989), no. 4, 576. Acta Math. Sinica (N.S.) \textbf{4} (1988), no. 3, 250--265.

\bibitem{Tian14} \textit{G. Tian,} {A third derivative estimate for conic Monge-Ampere equations,} preprint.

\bibitem{Tosatti} \textit{V. Tosatti,} {A general Schwarz lemma for almost-Hermitian manifolds,} Comm. Anal. Geom. \textbf{15} (2007), no. 5, 1063--1086.

\bibitem{TWWY} \textit{V. Tosatti, Y. Wang, B. Weinkove} and \textit{X. Yang,} {$C^{2,\alpha}$ estimate for nonlinear elliptic equations in complex and almost complex geometry,} Calc. Var. Partial Differential Equations \textbf{54} (2015), no. 1, 431--453.

\bibitem{TW10a} \textit{V. Tosatti} and \textit{B. Weinkove,} {Estimates for the complex Monge-Amp\`{e}re equation on Hermitian and balanced manifolds,} Asian J. Math. \textbf{14} (2010), no. 1, 19--40.

\bibitem{TW10b} \textit{V. Tosatti} and \textit{B. Weinkove,} {The complex Monge-Amp\`{e}re equation on compact Hermitian manifolds,} J. Amer. Math. Soc. \textbf{23} (2010), no. 4, 1187--1195.

\bibitem{TW13a} \textit{V. Tosatti} and \textit{B. Weinkove,} {The Monge-Amp\`{e}re equation for $(n-1)$-plurisubharmonic functions on a compact K\"{a}hler manifold,} preprint, arXiv:1305.7511.

\bibitem{TW13b} \textit{V. Tosatti} and \textit{B. Weinkove,} {Hermitian metrics, $(n-1,n-1)$ forms and Monge-Amp\`{e}re equations,} preprint, arXiv:1310.6326.

\bibitem{TWY} \textit{V. Tosatti, B. Weinkove} and \textit{S.-T. Yau,} {Taming symplectic forms and the Calabi-Yau equation,} Proc. Lond. Math. Soc. (3) \textbf{97} (2008), no. 2, 401--424.

\bibitem{Wang} \textit{Y. Wang,} {On the $C^{2,\alpha}$-regularity of the complex Monge-Amp\`{e}re equation,} Math. Res. Lett. \textbf{19} (2012), no. 4, 939--946.

\bibitem{Weinkove04} \textit{B. Weinkove,} {The J-Flow, the Mabuchi energy, the Yang-Mills flow and Multiplier Ideal Sheaves,} PhD thesis, Columbia University, 2004.

\bibitem{Weinkove07} \textit{B. Weinkove,} {The Calabi-Yau equation on almost-K\"{a}hler four-manifolds,} J. Differential Geom. \textbf{76} (2007), no. 2, 317--349.

\bibitem{Yau} \textit{S.-T. Yau,} {On the Ricci curvature of a compact K\"ahler manifold and the complex Monge-Apm\`ere equation I,} Comm. Pure Appl. Math. \textbf{31} (1978), no. 3, 339--411.

\bibitem{Zhang} \textit{X. Zhang,} {A priori estimates for complex Monge-Amp\`{e}re equation on Hermitian manifolds,} Int. Math. Res. Not. IMRN 2010, no. 19, 3814--3836.

\end{thebibliography}
\end{document}